\documentclass[12pt,reqno]{amsart}
\usepackage{amssymb}
\usepackage{amsmath}
\usepackage{amscd}
\usepackage{enumitem}
\usepackage{mathrsfs}
\usepackage[all]{xy}
\usepackage{hyperref}
\usepackage{mathtools}

\usepackage[margin=1.25in]{geometry}

\RequirePackage{mathrsfs} \let\mathcal\mathscr

\newtheorem{theorem}{Theorem}[section]
\newtheorem{lemma}[theorem]{Lemma}
\newtheorem{proposition}[theorem]{Proposition}
\newtheorem{corollary}[theorem]{Corollary}
\newtheorem{conjecture}[theorem]{Conjecture}

\theoremstyle{definition}
\newtheorem*{ack}{Acknowledgements}
\newtheorem{remark}[theorem]{Remark}
\newtheorem{example}[theorem]{Example}

\newtheorem*{spec-todo}{Speculative TO DO  list}
\newtheorem{condition}[theorem]{Condition}

\numberwithin{equation}{section} \numberwithin{figure}{section}

\DeclareMathOperator{\Pic}{Pic} 
\DeclareMathOperator{\Gal}{Gal} 
 \DeclareMathOperator{\Spec}{Spec}
 
\DeclareMathOperator{\Hom}{Hom} 
\DeclareMathOperator{\im}{Im}   
 
 \DeclareMathOperator{\Val}{Val}

\DeclareMathOperator{\Br}{Br}

 \DeclareMathOperator{\Norm}{N}
\DeclareMathOperator{\coker}{coker}

\newcommand{\kbar}{{\overline{k}}}

\newcommand{\Adele}{\mathbf{A}}

\newcommand\FF{\mathbb{F}}
\newcommand\PP{\mathbb{P}}

\newcommand\ZZ{\mathbb{Z}}
\newcommand\NN{\mathbb{N}}
\newcommand\QQ{\mathbb{Q}}
\newcommand\RR{\mathbb{R}}

\newcommand\GG{\mathbb{G}}

\newcommand\Gm{\GG_\mathrm{m}}
\newcommand\OO{\mathcal{O}}

\renewcommand{\leq}{\leqslant}
\renewcommand{\le}{\leqslant}
\renewcommand{\geq}{\geqslant}
\renewcommand{\ge}{\geqslant}
\newcommand{\x}{\mathbf{x}}

\newcommand{\bmu}{\boldsymbol{\mu}}

\newcommand{\A}{\mathbf{A}}

\newcommand{\y}{\mathbf{y}}
\renewcommand{\u}{\mathbf{u}}
\renewcommand{\v}{\mathbf{v}}
\newcommand{\cA}{\mathcal{A}}
\newcommand{\ve}{\varepsilon}
\newcommand{\fo}{\mathfrak{o}}
\newcommand{\fp}{\mathfrak{p}}
\newcommand{\fa}{\mathfrak{a}}
\DeclareMathOperator{\n}{N}

\newcommand{\mmu}{\boldsymbol{\mu}}
\renewcommand{\H}{\mathrm{H}}
\newcommand{\ff}[1]{{\kappa(#1)}}

\newcommand{\Zn}{\ZZ/n}

\newcommand{\QZ}{\QQ/\ZZ}
\newcommand{\sX}{\mathcal{X}}
\newcommand{\sY}{\mathcal{Y}}
\newcommand{\sD}{\mathcal{D}}
\newcommand{\sV}{\mathcal{V}}

\newcommand{\sB}{\mathcal{B}}

\newcommand{\sP}{\mathcal{P}}
\newcommand{\p}{\mathfrak{p}}
\newcommand{\R}{\mathrm{R}}
\newcommand{\et}{\mathrm{\acute{e}t}}
\DeclareMathOperator{\spec}{sp}

\newcommand{\ns}{\mathrm{ns}}
\newcommand{\sm}{\mathrm{sm}}
\DeclareMathOperator{\Cl}{Cl}
\DeclareMathOperator{\cl}{cl}

\title{Failures of  weak approximation in families}

\author{M. Bright}
\address{Mathematisch Instituut\\Universiteit Leiden\\Niels Bohrweg 1\\2333 CA Leiden\\Netherlands}
\email{m.j.bright@math.leidenuniv.nl}

\author{T.D. Browning}
\address{School of Mathematics\\
University of Bristol\\ Bristol\\ BS8 1TW\\ United Kingdom}
\email{t.d.browning@bristol.ac.uk}

\author{D. Loughran}
\address{Leibniz Universit\"{a}t Hannover,
Institut f\"{u}r Algebra, Zahlentheorie
    und Diskrete Mathematik\\
Welfengarten 1\\
30167 Hannover\\
Germany.}
\email{loughran@math.uni-hannover.de}

\date{\today}

\begin{document}

\begin{abstract}
Given a family of varieties $X\to \PP^n$ over a number field, we determine conditions under which 
there is a Brauer--Manin obstruction to weak approximation for $100\%$ of the fibres which are everywhere locally
soluble.
\end{abstract}

\maketitle
\tableofcontents

\thispagestyle{empty}

\section{Introduction}

This paper is concerned with the Hasse principle and weak approximation for families of varieties defined over a number field $k$.
Given a smooth projective geometrically integral variety $X$ over $k$ we have the obvious
diagonal embedding
$$
X(k)\rightarrow X(\mathbf{A}_k)=
\prod_{\textrm{all places }\nu} X(k_\nu),
$$
where $X(\mathbf{A}_k)$ is the set of ad\`eles of $X$.  Recall that 
a class $\mathcal{X}$ of 
 smooth projective geometrically integral varieties $X$  over
$k$ is said to satisfy the {\em Hasse principle} if $X(k)\neq \emptyset$ whenever
$X(\mathbf{A}_k)\neq \emptyset$. Likewise $\mathcal{X}$
is said to satisfy {\em weak approximation} if 
for any $X$ in $\mathcal{X}$,
the image of $X(k)$ in $X(\Adele_k)$ is dense with respect to the
product topology. Note that with these definitions,
weak approximation holds for $X$ if $X(\mathbf{A}_k) = \emptyset$.
In 1970, Manin used 
the Brauer group $\Br X=\H_{\text{\'et}}^2(X,\mathbb{G}_m)$ to define the {\em Brauer set} 
$X(\A_k)^{\Br}\subset X(\A_k)$. 
This set contains the rational points of $X$ and can sometimes be used to obstruct the Hasse principle or weak approximation (see e.g.~\cite[\S5]{Sko01}).
It has been conjectured by Colliot-Th\'el\`ene (see \cite{bud}) that 
this {\em Brauer--Manin obstruction} is the only obstruction to the Hasse principle or weak approximation for any 
 smooth projective geometrically integral variety $X$  over
$k$ which is geometrically rationally connected. Here we are interested in
the following special case of Colliot-Th\'el\`ene's conjecture.

\begin{conjecture}[Colliot-Th\'el\`ene]\label{con}
Let $X$  be a smooth projective geometrically integral variety over a number field $k$, which is 
geometrically rationally connected, such that  $X(\A_k)\neq \emptyset$ and 
$ \Br X/\Br k=0$. Then $X(k)\neq \emptyset$ and $X$ 
satisfies  weak approximation.
\end{conjecture}

%Here we abuse notation and let $\Br X/\Br k= \coker (\Br k \to \Br X)$.
%Note that the condition $X(\A_k)\neq\emptyset$ implies that the natural map $\Br k \to \Br X$ is injective, so the notation $\Br %X/\Br k$ makes sense.
%We have  $\Br X/\Br k=\H^1(k, \Pic \bar X)$ when $X(\A_k)\neq \emptyset$. 
There are many examples in the literature where
the  converse of Conjecture~\ref{con} holds, and in this 
paper we are motivated by the extent to which this converse  holds in general. 
For example, consider
the  Ch\^atelet surface $X$ 
which arises over $\QQ$ as a smooth proper model of the affine surface
 $$
y^2+z^2=L_1(t)L_2(t)L_3(t)L_4(t),
$$
where $L_1,\dots,L_4\in \QQ[t]$ are pairwise non-proportional linear polynomials. Then  
$\Br X / \Br \QQ \cong(\ZZ/2\ZZ)^2$ and $X(\QQ)\neq \emptyset$, and it follows from 
\cite[Rem.~3.4]{annal} that $X$ fails weak approximation.
For varieties whose dimension is large compared to the degree, the 
 Hardy--Littlewood circle method can be used to show that the converse of Conjecture~\ref{con} is true.
 Recent work of Browning and Heath-Brown
\cite[Thm.~1.1]{many}, for example, shows that the 
Hasse principle and weak approximation holds for any 
smooth geometrically integral 
variety  $X\subset{\PP}^m$  over 
 $\QQ$ for which 
$$\dim(X)\geq (\deg(X)-1)2^{\deg(X)}-1.$$
Note that here one has
$ \Br X / \Br \QQ=0$
(see  \cite[\S 1]{many}).

On the other hand, it is possible to construct examples of varieties $X$ over $k$ covered by 
Colliot-Th\'el\`ene's conjecture for which 
$ \Br X / \Br k \neq 0$
and yet $X$ satisfies the Hasse principle and weak approximation. 
The following example is due to Colliot-Th\'el\`ene and Sansuc \cite[Ex.~D, p.~223]{CTS77}.

\begin{example}
Let $K/k$ be a biquadratic extension and 
let $\mathbf{N}_{K}\in k[x_1,\dots,x_4]$ be the associated {\em norm form}. 
Then the equation $\mathbf{N}_K(x_1,\dots,x_4)=1$ defines a $k$-torus $T$. 
Let $X$ be a smooth compactification of $T$. 
Then we have $X(k)\neq \emptyset$ and $\Br X / \Br k\cong\ZZ/2\ZZ$.
However, as explained in \cite[Ex.~D, p.~223]{CTS77}, if all decomposition groups of $K/k$ are cyclic
then weak approximation holds for $T$, and so also for  $X$ (this occurs for $\QQ(\sqrt{13}, \sqrt{17})/\QQ$, 
for example).
\end{example}
A further  example involving del Pezzo surfaces of degree four can be found in recent work of Jahnel and Schindler
\cite[Ex.~4.3]{J-S}. 

These examples illustrate that we cannot expect the converse to hold in every case. Instead, 
we shall establish a weak converse to Conjecture \ref{con}, by considering this problem
in families. Namely, suppose that we are given a   family $X\to \PP^n$ over a number field $k$, with  $X$ non-singular projective and  geometrically integral. Let $K$ be  the function field of $\PP^n$ and assume that 
the   generic fibre $X_\eta$ is such that either
$\H^1(K,\Pic X_{\bar \eta})$ or $\Br X_\eta/\Br K=\coker(\Br K \to \Br X_\eta)$ is non-trivial. Under suitable hypotheses, our main result
shows that  with probability $1$  any element of the family which is everywhere locally soluble
fails weak approximation.

\medskip

An important first step in our investigation is to understand the proportion of elements in a family of varieties 
that are everywhere locally soluble. 
Let $\pi:X\to \PP^n$ be a dominant projective $k$-morphism with geometrically integral generic fibre.
Let $H:\PP^n(k)\to \RR_{\geq 1}$ be the standard exponential height (see \S 
\ref{s:goat-herder} for its definition). Then we let
\begin{equation}\label{eq:horse}
\sigma(\pi) = \lim_{B \to \infty} \frac{\#\{ P \in \PP^n(k): H(P) \leq B,~
X_P(\Adele_k) \neq \emptyset\}}{\#\{ P \in \PP^n(k): H(P) \leq B\}},
\end{equation}
if the limit exists, where $X_P=\pi^{-1}(P)$.
This is the proportion of varieties in the family which are everywhere locally soluble.

Recall that a scheme over a field is said to be \emph{split} if it contains a geometrically
integral open subscheme. 
Bearing this in mind we will establish the following result.

\begin{theorem} \label{thm:Ekedahl}
	Let $k$ be a number field.
	Let $\pi:X\to \PP^n$ be a dominant quasi-projective $k$-morphism with geometrically integral generic fibre.
	Assume that:
	 \begin{enumerate}
	 \item	the fibre of $\pi$ above each codimension $1$ point of $\PP^n$ is split,
	 %\item	For all but finitely many places $\nu$ of $k$, the fibre of $\pi$ at each codimension $1$ point of $\PP^n$ is split,
	 \item 	$X(\Adele_k) \neq \emptyset$.
	\end{enumerate}
	 Then $\sigma(\pi)$ exists, is non-zero and is equal to a product of local densities.
\end{theorem}

This result proves a special case of a conjecture due to  Loughran \cite[Con.~$1.7$]{Loughran}.
A precise statement about the  shape of the leading constant is recorded in 
 Theorem~\ref{thm:Ekedahl2}.

Theorem \ref{thm:Ekedahl} should be compared with  work of Poonen and Voloch \cite[Thm.~3.6]{PV}, which is concerned with  
 the family $X\to \PP^n$ of all hypersurfaces of degree $d$ in $\PP^m$ over $\QQ$, where $n=\binom{m+d}{d}-1$
 and $d,m\geq 2$ are such that $(d,m)\neq (2,2)$.
Our work  generalises this to number fields, but also
applies, for example, to  the 
family of all diagonal hypersurfaces of  degree $d$ and  dimension exceeding $1$.
Hypothesis (1) in Theorem \ref{thm:Ekedahl} generalises the condition 
$(d,m)\neq (2,2)$ in the work of Poonen and Voloch. 
Indeed, consider the family 
$$
X=\left\{\sum_{0\leq i\leq j\leq 2}a_{i,j}x_ix_j=0\right\}\subset \PP^5\times \PP^2
$$
of all conics over $\QQ$ in $\PP^2$.  
Serre
\cite{serre} has shown that $\sigma(\pi)=0$ for this family and one easily checks that 
the fibre over the generic point of the discriminantal hypersurface
$$\Delta(a_{0,0},\dots,a_{2,2})=0$$
is not split, being a union of two lines which are conjugate over
a quadratic field extension. 
Hypothesis (2) is also clearly necessary
in the statement of Theorem~\ref{thm:Ekedahl}.  The result is proved using the sieve of Ekedahl \cite{Eke91}
together with Deligne's work on the Weil conjectures \cite{deligne}, and is carried out in \S \ref{s:Ekedahl}.

We are also able to obtain a version of Theorem \ref{thm:Ekedahl} for integral points on affine
spaces. For a  number field $k$ over $\QQ$, let $\mathfrak{o}$ be its ring of integers and 
 let $k_\infty=\mathfrak{o}\otimes_\ZZ \RR$ be the associated  commutative $\RR$-algebra.
To state the result, we say that an open subset $\Theta \subset V$ of a finite dimensional real vector space $V$ is a \emph{semi-cone} if
$B\Theta \subset \Theta$ for all $B\geq 1$.

\begin{theorem} \label{thm:Ekedahl_affine}
	Let $k$ be a number field.
	Let $\pi:X\to \mathbb{A}^n$ be a dominant quasi-projective $k$-morphism with geometrically integral generic fibre.
	Assume that:
	 \begin{enumerate}
	 \item	the fibre of $\pi$ above each codimension $1$ point of $\mathbb{A}^n$ is split,
	 \item 	$X(\Adele_k) \neq \emptyset$,
	 \item $\pi(X(k_\nu)) \subset k_\nu^n$ contains a semi-cone for every real place $\nu$ of $k$.
	\end{enumerate}
Let $\Psi\subset k_\infty^{n}$ be a  bounded subset of positive measure that lies in some semi-cone contained in $\pi(X(k_\infty))$ and whose boundary has measure zero.
Then the limit 
	$$
	\lim_{B \to \infty} \frac{\#\{ P \in \fo^n \cap B\Psi: ~X_P(\Adele_k) \neq \emptyset\}}
	{\#\{ P \in \fo^n \cap B\Psi\}},
$$
	exists, is non-zero and is equal to a product of local densities.
\end{theorem}
Note that some form of ``unboundedness'' assumption at the real places, 
such as Condition~(3), is clearly necessary for the conclusion
to hold. Indeed, the conclusion of Theorem \ref{thm:Ekedahl_affine}
does not hold if, for example, $\pi(X(k_\nu))$ is bounded for some real place $\nu$.
%,is clearly necessary in order for the the above limit to be non-zero.
Such unboundedness conditions naturally arise in the study of integral points and fibrations;
see work of Derenthal and Wei \cite{DW}, for example.

Now let $X$ be as in Theorem \ref{thm:Ekedahl}. A simple consequence
of this result is the fact that there exists a member of the
family which is everywhere locally soluble. In particular, this allows us to deduce
the following corollary.

\begin{corollary}
	Let $k$ be a number field.
	Let $\pi:X\to \PP^n$ be a dominant quasi-projective $k$-morphism with geometrically integral generic fibre.
	Assume that:
	 \begin{enumerate}
	 \item	the fibre of $\pi$ above each codimension $1$ point of $\PP^n$ is split,
	 \item 	the smooth fibres of $\pi$ satisfy the Hasse principle.
	\end{enumerate}
	Then $X$ satisfies the Hasse principle.
\end{corollary}
Whilst this result does not explicitly occur in the literature, it has long been known in some form to experts
(cf. \cite[Thm.~2.1]{Sko96} and  \cite[Thm.~3.2.1]{Har97}). 
Our Theorem~\ref{thm:Ekedahl} may be viewed as a quantitative refinement of 
this fact.

\medskip

Suppose now that $\pi:X\to\PP^n$ is a family over $k$ satisfying  the hypotheses of Theorem 
\ref{thm:Ekedahl} and whose generic fibre is smooth. Let $K$ be the function field of $\PP^n$ and put  $\eta: \Spec(K) \to \PP^n$ for the generic point. 
Our hypotheses imply that fibres of $\pi$ are smooth and geometrically integral over some dense open subset of $\PP^n$.
One would like to study the Brauer--Manin obstruction to weak approximation for these fibres,
for which one requires a non-constant Brauer group element on these fibres. 
It is natural to try and achieve this by  assuming that the generic fibre
has non-constant Brauer group. A significant problem with this, however, is that
there are  families of varieties which we would like to address where this does not happen.
For example, if $k$ contains a primitive third root of unity, then Uematsu~\cite[Thm.~5.1]{uematsu} has shown
that the generic diagonal cubic surface over $K$ has constant Brauer group.

One of the chief novelties of our investigation is that one can work directly with non-trivial elements of the group
$\H^1(K, \Pic X_{\bar{\eta}})$.  Here $\bar\eta \colon \Spec(\bar{K}) \to \PP^n$ denotes a geometric point lying above $\eta$.  Note that the Hochschild--Serre spectral sequence (see e.g.~\cite[\S 2]{uematsu})
yields the exact sequence
$$\Br K \to \Br_1 X_\eta \to \H^1(K, \Pic X_{\bar{\eta}}) \to \H^3(K, \Gm)$$
where $\Br_1 X_\eta = \ker(\Br X_\eta \to \Br X_{\bar{\eta}}).$
In particular, elements of $\H^1(K, \Pic X_{\bar{\eta}})$ do not necessarily lift to $\Br_1 X_\eta$, due to a possible obstruction lying in $\H^3(K, \Gm)$.
For those $P \in \PP^n(k)$ such that $X_P$ is smooth, however, such an element specialises to give an element of $\H^1(k, \Pic X_{\bar P})$. Since  $\H^3(k, \Gm)=0$, such elements \emph{do} lift to an element of $\Br_1 X_P$, with
which one can  try to obstruct weak approximation on $X_P$.  
Our strategy is to show that, under appropriate geometric hypotheses,  almost 
 surely a non-trivial  element
of  $\H^1(K, \Pic X_{\bar{\eta}})$ or $\Br X_\eta/\Br K$   gives an obstruction to weak approximation on $X_P$,
if $X_P$ is everywhere locally soluble.
Our main result in this direction is as follows.

\begin{theorem}\label{thm:H1}
Let $k$ be a number field and let 
 $\pi \colon X \to \PP^n$ be a flat, surjective $k$-morphism of finite type, 
with $X$ smooth, projective and geometrically integral over $k$.
Let $\eta \colon \Spec K \to \PP^n$ denote the generic point
and suppose that the generic fibre $X_\eta$ is geometrically connected.
Assume the following hypotheses:
\begin{enumerate}[label=(\arabic*)]
\item\label{hyp2} $X(\A_k)\neq \emptyset$;
\item\label{hyp1} the fibre of $\pi$ at each codimension $1$ point of $\PP^n$ is geometrically integral;
\item\label{hyp3} the fibre of $\pi$ at each codimension $2$ point of $\PP^n$ has a geometrically reduced component;
\item\label{hyp7} $\H^1(k, \Pic\bar{X})=0$;
\item\label{hyp5} $\Br \bar{X}=0$;
\item\label{hyp6}
$\H^2(k,\Pic \PP_{\bar k}^n)\to \H^2(k,\Pic \bar X)$ is injective;
\item\label{hyp4}
either $\H^1(K,\Pic  X_{\bar \eta})\neq 0$ or $\Br X_\eta / \Br K \neq 0$.
\end{enumerate}
Then %, with polynomial decay,  %T
the fibre $X_P$ is smooth and 
fails weak approximation for $100\%$ of rational points $P\in \PP^n(k)\cap \pi(X(\A_k))$.
\end{theorem}

On 
letting $S=\PP^n(k)\cap \pi(X(\A_k))$, 
the conclusion of the theorem means that 
%for the standard height function $H:\PP^{n}(k)\to \RR_{\geq 1}$  %T
%there exists $\delta>0$ such that  %T
$$
\limsup_{B\rightarrow \infty} 
%B^\delta\cdot 
\frac{\#\{P\in S: \mbox{$H(P)\leq B$ and  $X_P$ satisfies weak approximation}\}}{
\#\{P\in S: H(P)\leq B\}}
= 0.
$$
The proof of Theorem 
\ref{thm:H1} is carried out in \S\S\ref{s:cow-herder}--\ref{s:sieve}.

Let us briefly  explain how the conditions given in Theorem~\ref{thm:H1} arise.
First note that any family $X\to \PP^n$ satisfying the hypotheses of 
Theorem~\ref{thm:H1} will automatically 
satisfy those  of Theorem  \ref{thm:Ekedahl}. Hence a positive proportion
of the varieties in the family are everywhere locally soluble.
%Condition~\ref{hyp1} is  a substantial strengthening of  the corresponding condition in 
%Theorem~\ref{thm:Ekedahl}. While we expect that the fibres $X_P$ should still  fail
% weak approximation for $100\%$ of points $P\in \PP^n(k)\cap \pi(X(\A_k))$ if this condition is weakened, 
% we will provide an example in  \S \ref{sec:DP4} which shows that in general  this will no longer be true with  polynomial decay.  %T [I removed this, but maybe we do need to say something about how 
%Condition~\ref{hyp1} is  a  strengthening of  the corresponding condition before]
As already explained, Condition \ref{hyp4} is a natural condition to impose
so that one actually has a Brauer group element on the fibres to work with. The remaining
conditions in Theorem~\ref{thm:H1} are there to guarantee that one can actually use this element to obstruct weak
approximation. 
Note that, assuming Condition~\ref{hyp5}, Condition~\ref{hyp7} is equivalent to the condition $\Br X = \Br k$.
Condition \ref{hyp6} is perhaps the least natural looking  in the theorem; it will be discussed further
in \S \ref{s:bee-keeper}.
It holds, for example, if the map $\pi$ admits a section. 
In Proposition \ref{prop:ci}, furthermore,  we shall show that Conditions~\ref{hyp7}, \ref{hyp5} and~\ref{hyp6} hold whenever $X\subset \PP^r\times \PP^n$ is a complete intersection of dimension $\geq 3$.

One of the main tools in the study of the Brauer group is Grothendieck's purity theorem.
This implies  that if $Y$ is a smooth projective variety over a
field $k$ of characteristic $0$ satisfying $\Br Y = \Br k$, then any non-constant Brauer group element
defined on some open subset $U \subset Y$ must be ramified somewhere on $Y$.
One of the challenges of this paper is to define the residue of an element of $\H^1(k,\Pic  Y_{\bar k})$ and to prove an analogue of the purity theorem in this setting, under suitable assumptions.
This need to develop so much theory from scratch is a partial explanation for the  length of the proof of 
Theorem \ref{thm:H1}. In this setting, for $\pi:X\to \PP^n$,  
Conditions~\ref{hyp7} and~\ref{hyp6} are  necessary in the sense that the failure of either of them gives rise to non-zero unramified elements of $\H^1(K, \Pic X_{\bar{\eta}})$; such elements may or may not give an obstruction to weak approximation on the fibres, but the methods of this article say nothing about them.
Conditions~\ref{hyp3} and~\ref{hyp5} are sufficient to ensure that this is the only way such unramified elements of $\H^1(K, \Pic X_{\bar{\eta}})$ can arise, but we do not know whether they are also necessary.
Once one has these conditions, we use them to deduce that 
any non-zero element of $\H^1(K,\Pic  X_{\bar \eta})$, or any non-zero element of $\Br X_\eta / \Br K$, must be ramified along some
fibre above a divisor $D \subset \PP^n$. We then find conditions so that
if the reduction of a rational point $P\in \PP^n(k)$ modulo some
prime $\fp$ of $k$ lies in $D$, then $X_P$ has a suitable kind of bad reduction at~$\fp$ which forces a Brauer--Manin obstruction to weak approximation for $X_P$ to occur at~$\fp$. Once we have this criterion, 
we then prove Theorem~\ref{thm:H1} by showing that this type of bad reduction occurs
$100\%$ of the time, using the large sieve.

\medskip

In \S \ref{s:examples} we shall collect together some examples to illustrate our main results. 
Foremost among those, which we explicate here,
is the smooth biprojective hypersurface
\begin{equation}\label{eq:onion-collector}
X=\left\{ a_0x_0^3+a_1x_1^3+a_2x_2^3+a_3x_3^3=0 \right\} \subset \PP^3\times \PP^3,
\end{equation}
over a number field $k$,  
viewed as the family of diagonal cubic surfaces $X_P\subset \PP^3$.  It is easy to see that 
$X$ is $k$-rational and so Conditions~\ref{hyp2}, \ref{hyp7} and~\ref{hyp5} are automatically  met in Theorem~\ref{thm:H1}.
Condition~\ref{hyp6} follows from 
Proposition \ref{prop:ci}.
For a given rational point  $P=(a_0,\dots,a_3)\in \PP^3(k)$ the fibre $X_P$ is smooth whenever the product $a_0a_1a_2 a_3$ is non-zero.
The locus of bad reduction consists of the union of the four coordinate planes $ \{a_i = 0\}$ in $\PP^3$. Above each of these divisors the generic fibre is a cone over a smooth cubic curve, and hence is geometrically integral. Worse reduction happens over the intersections of these divisors, that is, where more than one of the $a_i$ vanish.
However, one notes that the fibre over any point of codimension $2$ has a geometrically reduced component.
Thus Conditions~\ref{hyp1} and~\ref{hyp3} in 
Theorem~\ref{thm:H1} are also satisfied. 
Colliot-Th\'el\`ene, Kanevsky and Sansuc
\cite{ct-k-s} have carried out an extensive investigation of the Brauer--Manin obstruction for 
the smooth fibres $X_P$. 
Under the assumption that the number field $k$ contains a primitive cube root of unity, 
it follows from their work that 
$
\H^1(K, \Pic X_{\bar \eta})\cong\ZZ/3\ZZ.
$ In Lemma \ref{lem:cubic_H1} we shall show that this continues to hold over any number field $k$.
This takes care of Condition~\ref{hyp4} and so all of the hypotheses of Theorem \ref{thm:H1} are met. Thus, when ordered by height,  
$100\%$ of diagonal cubic surfaces over $k$ which are everywhere locally soluble 
fail weak approximation. Since the hypotheses of Theorem \ref{thm:Ekedahl} are also satisfied, it is possible 
to calculate the precise proportion of $X_P$ which are everywhere locally soluble.
 This will be carried out in  \S \ref{s:fish-fryer} 
 when $k=\QQ$
 and 
leads to the following rather succinct result.

\begin{theorem}\label{thm:cubics}
Approximately
$86\%$ of diagonal cubic surfaces over $\QQ$, when ordered by height,  fail weak approximation. 
\end{theorem}

It is worth emphasising (see  \cite[Prop.~3.4]{PV})
 that the analogous result for the full family 
$X\to \PP^{19}$ 
of {\em all} cubic surfaces is presumed to be false. Indeed, since $\H^1(k,\Pic \bar V)=0$ for a randomly chosen cubic surface $V\subset \PP^3$ over a number field $k$, it follows from Conjecture \ref{con} that the Hasse principle and weak approximation \emph{hold} for $100\%$ of the surfaces which are everywhere locally soluble.

\medskip

We close by comparing Theorem \ref{thm:H1} with far-reaching work of 
Harari \cite{harari} on the fibration method.
Specialising his work to 
families given by 
a surjective morphism $\pi:X\to \PP^n$ over a number field $k$,  with $X$ a smooth and  geometrically integral
variety over $k$, 
we make the following 
hypotheses:
\begin{itemize}
\item
the generic fibre $X_\eta$  has a $K$-point,
\item
$X(\A_k)\neq \emptyset$,
\item
$\Br X /\Br k=0$,

\item
$\Br X_\eta /\Br K\neq 0$.
\end{itemize}
Then it follows from \cite[Prop.~6.1.2]{harari} that there is a Zariski dense set of points $P\in \PP^n(k)$ such that 
the fibre $X_P$ fails weak approximation.
Our result provides finer quantitative information about the frequency of counter-examples to weak approximation, does not require $\pi$  to have a section, and works with non-trivial elements of $\H^1(K, \Pic X_{\bar{\eta}})$ regardless of whether they lift to $\Br X_\eta$.

The plan of the paper is as follows. In
\S \ref{s:examples} we will  complete the proof of Theorem~\ref{thm:cubics}
 and  collect together further examples to illustrate Theorem \ref{thm:H1}.
 In \S \ref{s:Ekedahl} we will prove Theorem~\ref{thm:Ekedahl}. This part of the paper is completely self-contained.
In \S \ref{s:cow-herder} we  give an overview of the proof of 
 Theorem \ref{thm:H1}. The algebraic part is the object of  
 \S \ref{sec:Pic} and the analytic part is handled in \S \ref{s:sieve}.

\begin{ack}
The authors are grateful to Jean-Louis Colliot-Th\'{e}l\`{e}ne,
Bas Edixhoven, Christopher Frei and David Holmes for useful discussions.
Thanks are also due to the anonymous referees for several helpful comments. 
While working on this paper the 
second author
was  supported by {\em ERC grant} \texttt{306457}. 
\end{ack}

\section{Examples}
\label{s:examples}

\subsection{Diagonal cubic surfaces}\label{s:fish-fryer}
The aim of this section is to prove Theorem \ref{thm:cubics}.
We begin by verifying that the Picard group of the generic diagonal cubic surface has non-trivial
first cohomology group. The proof runs along very similar lines to the proof of \cite[Prop.~6.1]{p-t1},
so we shall be brief on the details. 

\begin{lemma} \label{lem:cubic_H1}
	Let $k$ be a field of characteristic zero and let 
	$$S: \quad x_0^3 + a_1x_1^3+ a_2x_2^3+ a_3x_3^3 = 0,$$ 
	over the function field $K=k(a_1,a_2,a_3)$. Then
	$\H^1(K, \Pic S_{\bar K}) \cong \ZZ/3\ZZ.$
\end{lemma}
\begin{proof}
	If $\bmu_3 \subset k$, then the result
	follows from \cite[Prop.~1]{ct-k-s}. (This result is stated only
	for number fields and local fields, but the proof works in our more general
	setting.) Thus we assume that $\bmu_3 \not \subset k$. Let $L = k(\theta)$,
	where $\theta$ is a primitive third root of unity. 
	The splitting field of $S$ is 
	$k(\theta, \sqrt[3]{a_1}, \sqrt[3]{a_2}, \sqrt[3]{a_3})$,
	whose Galois group is isomorphic to $(\ZZ/3\ZZ)^3 \rtimes \ZZ/2\ZZ$, where the non-trivial
	element of $\ZZ/2\ZZ$ acts as multiplication by $-1$ on $(\ZZ/3\ZZ)^3$.
	As in the proof of \cite[Prop.~6.1]{p-t1}, inflation--restriction yields
	an isomorphism
	$$\H^1(K, \Pic S_{\bar K}) \cong \H^1(K(\theta), \Pic S_{\bar K})^{\Gal(L/k)}.$$
	To prove the result, it suffices to show that the group $\Gal(L/k)$ acts trivially on 
	$\H^1(K(\theta), \Pic S_{\bar K})$. This is proved in a similar manner to
	\cite[Prop.~6.1]{p-t1}, on using the explicit description of the above cohomology
	group given in \cite[Lem.~2]{ct-k-s}.
\end{proof}

Lemma \ref{lem:cubic_H1}, once combined with the discussion given in the introduction, confirms 
that the family \eqref{eq:onion-collector} of all diagonal cubic surfaces over any number field $k$ meets the hypotheses of Theorems~\ref{thm:Ekedahl} and \ref{thm:H1}.

We now explicitly calculate the proportion $\sigma$
of diagonal cubic surfaces over $\QQ$ which are everywhere locally soluble,
in order to prove Theorem \ref{thm:cubics}.
It follows from Theorem \ref{thm:Ekedahl} that we have an 
Euler product $\sigma = \sigma_\infty \prod_p \sigma_p$
of local densities (see \eqref{eqn:euler_product} for a precise description).
Clearly $\sigma_\infty = 1$, and moreover we shall soon see that $\sigma_p=1$ if $p \equiv 2 \pmod 3$ (since then every element of $\FF_p$ is a cube).
The main result here is the following.
\begin{theorem}\label{thm:dog-handler}
	The proportion of diagonal cubic surfaces which are everywhere locally soluble
	is
	$$\sigma = \sigma_3 \prod_{p \equiv 1 \bmod 3} \sigma_p,$$
	where 
	$$\sigma_p =
	\begin{cases}
	\frac{(1 - 1/p)^3}{(1-1/p^3)^3}
	\left(1 - \frac{1}{p} + \frac{1}{p^2}\right)
	\left(1 + \frac{1}{p} + \frac{1}{3 p^2}\right)
	\left(1 + \frac{3}{p} + \frac{3}{p^2}\right), & p \equiv 1 \bmod 3, \\
	
	\frac{(1 - 1/p)^3}{(1-1/p^3)^3}
	\left(1 + \frac{3}{p} + \frac{46}{9p^2} + \frac{7}{p^3}
	 + \frac{62}{9p^4} + \frac{19}{9p^5} + \frac{1}{p^6}\right), & p =3.
	\end{cases}$$
\end{theorem}

%We may write $\sigma_p$ in uniform way for all $p\neq 3$
%as
%\begin{align*}
%	\sigma_p = &\left(1 - \frac{1}{p}\right)^{3}
%    \left(1 - \frac{1}{p^3}\right)^{-3}
%	\left(1 - \frac{\chi(p)}{p} + \frac{1}{p^2}\right) \\
%	&\cdot \left(1 + \frac{1}{p} + \frac{1}{(2 + \chi(p))p^2}\right)
%	\left(1 + \frac{(2 + \chi(p))}{p} + \frac{(2 + \chi(p))}{p^2}\right).
%\end{align*}
%where $\chi$ denotes the unique non-principal Dirichlet
%character modulo $3$. 
%T [removed]
Here $\sigma$ has approximate value $0.860564$ (as one confirms on taking the product of all primes up to $10^6$).  This suffices 
to complete the proof of Theorem \ref{thm:cubics}.

To prove the theorem we use a criterion for testing locally solubility at each prime
given in \cite{ct-k-s}, then calculate the density of each case which occurs.
Our approach is inspired by the corresponding result for (non-diagonal) ternary cubic forms considered by Bhargava, Cremona and Fisher \cite{BCF}, who show that the 
density of plane cubic curves over $\QQ$ which are everywhere locally soluble is approximately $97\%$.

\subsubsection{An equivalence relation}

We define an equivalence relation on the set $\ZZ_{\geq 0}^4$
by declaring that
$$(\alpha_0,\alpha_1,\alpha_2,\alpha_3)\sim (\beta_0,\beta_1,\beta_2,\beta_3),$$
if and only if at least one of the following holds.
\begin{itemize}
	\item $(\alpha_0,\alpha_1,\alpha_2,\alpha_3)$ is a permutation of $(\beta_0,\beta_1,\beta_2,\beta_3).$
	\item There exists $k \in \ZZ_{\geq 0}$ such that 
	$$(\alpha_0,\alpha_1,\alpha_2,\alpha_3) = (\beta_0+k,\beta_1+k,\beta_2+k,\beta_3+k).$$
	\item $(\alpha_0,\alpha_1,\alpha_2,\alpha_3) \equiv (\beta_0,\beta_1,\beta_2,\beta_3) \mod 3.$
\end{itemize}
Given $\boldsymbol \alpha \in \ZZ_{\geq 0}^4$, we
define its \emph{weight} $w(\boldsymbol \alpha)$ to be the sum of its coordinates.
One easily checks that the following vectors give representatives of the
equivalence classes of $\sim$:
$$\boldsymbol{\delta}_1 = \mathbf{0}, ~\boldsymbol{\delta}_2 = (0,0,0,1), ~\boldsymbol{\delta}_3 = (0,0,0,2),~
\boldsymbol{\delta}_4 = (0,0,1,1), ~\boldsymbol{\delta}_5 = (0,0,1,2).$$

\subsubsection{Volumes in a special case}
Now fix  a prime $p$ with associated valuation $v$. We extend $v$ in the natural
way to a function $v:\ZZ_p^4 \to \ZZ_{\geq 0}^4$. For a vector $\mathbf{a}\in \ZZ_p^4$,
we shall denote by $S_{\mathbf{a}}$ the diagonal cubic surface with coefficients $\mathbf{a}$,
namely 
$$S_{\mathbf{a}} : \quad a_0x_0^3 +a_1x_1^3 +a_2x_2^3 +a_3x_3^3 = 0.$$
The equivalence relation $\sim$ has been chosen in such a way that if 
$\boldsymbol \alpha \sim \boldsymbol \beta$,
then 
\begin{align*}
	 p^{w(\boldsymbol \alpha)}\mu_p(\{\mathbf{a} \in \ZZ_p^4 &: v(\mathbf{a}) 
	= \boldsymbol \alpha, S_{\mathbf{a}}(\QQ_p) \neq \emptyset\}) \\
	&= p^{w(\boldsymbol \beta)}\mu_p(\{\mathbf{b} \in \ZZ_p^4 : v(\mathbf{b}) 
	= \boldsymbol \beta, S_{\mathbf{b}}(\QQ_p) \neq \emptyset\}),
\end{align*}
where $\mu_p$ denotes the usual Haar  measure
on $\ZZ_p^4$. We now calculate the relevant volumes for coefficient vectors
with valuations  $\boldsymbol{\delta}_1,\ldots,\boldsymbol{\delta}_5$.

\begin{lemma} \label{lem:3_A_i}
	For each prime $p$ and each $i \in \{1,\ldots,5\}$,
	there exists a rational number $\mathcal{A}_i$ such that 
	$$\mu_p(\{\mathbf{a} \in \ZZ_p^4 : v(\mathbf{a}) 
	= \boldsymbol{\delta}_i, S_{\mathbf{a}}(\QQ_p) \neq \emptyset\})
	= \mathcal{A}_i\cdot \mu_p(\{\mathbf{a} \in \ZZ_p^4 : v(\mathbf{a}) = \boldsymbol{\delta}_i\}).$$
	These are explicitly given in the following table.
	\begin{table}[ht]
	\centering
		\begin{tabular}{c|c|c|c|c|c}
			$p$&$\mathcal{A}_1$&$\mathcal{A}_2$&$\mathcal{A}_3$&$\mathcal{A}_4$&$\mathcal{A}_5$ \\
			\hline $3$&$1$&$1$&$7/9$&$1$&$1$ \\
			\hline $1 \bmod 3 $&$1$&$1$&$1$&$5/9$&$1/3$ \\
			\hline $2 \bmod 3 $&$1$&$1$&$1$&$1$&$1$ \\
		\end{tabular}
	\end{table}
\end{lemma}
\begin{proof}
	We use the criterion for solubility in $\QQ_p$ given
	in \cite[p.~28]{ct-k-s} (but see Remark~\ref{rem:p=3} for the case $p=3$).
	For $p\equiv 2 \bmod 3$ we have $\mathcal{A}_i=1$ for each $i$, as every such surface has a $\QQ_p$-point.
	For $p=3$ we similarly have $\mathcal{A}_i=1$ for $i\neq 3$. For $\boldsymbol{\delta}_3$,
	by Remark \ref{rem:p=3} we find that $\mathcal{A}_3=1-(48\cdot 6)/6^4=7/9$.

	Consider now the case $p\equiv 1 \bmod 3$. Here we have $\mathcal{A}_i=1$ for $i\not\in \{4,5\}$.
	For $\boldsymbol{\delta}_4$,
	there is a $\QQ_p$-point if and only if $-a_1/a_0$ is a cube, or 
	if $-a_1/a_0$ is not a cube but $-a_3/a_2$ is a cube.
	Thus $\mathcal{A}_4=1/3 + 2/9 = 5/9$ (this volume computation follows from 
	the simple generalisation of the fact
	that when $p \equiv 1 \bmod 3$, we have
	$\mu_p(\{ a \in \ZZ_p^*: a \text{ is a cube}\}) = (1/3)\cdot\mu_p(\{ a \in \ZZ_p^*\})$).
	For $\boldsymbol{\delta}_5$ there is a $\QQ_p$-point
	if and only if $-a_0/a_1$ is a cube, hence
	$\mathcal{A}_5=1/3$.
\end{proof}

\begin{remark} \label{rem:p=3}
	Let us take the opportunity to correct a small error in the criterion for solubility
	at the prime $p=3$ given in \cite[p.~28]{ct-k-s}. Namely,
	let $\mathbf{a} \in \ZZ_3^4$ with $v(\mathbf{a}) = (0,0,0,2)$.
	In \cite{ct-k-s} the authors claim that the surface $S_\mathbf{a}$
	has no $\QQ_3$-point if and only if $\mathbf{a}$ is congruent
	to $(1,2,4,0)$ modulo $9$, up to permuting coordinates. However this is clearly
	false; indeed the coefficient vector $(-1,2,4,9) \in \ZZ_3^4$ gives rise
	to a cubic surface over $\QQ_3$ with no $\QQ_3$-point.
	
	The correct condition is that the surface $S_\mathbf{a}$ has no $\QQ_3$-point
	if and only if $\mathbf{a}$ is congruent to $(1,2,4,0)$ modulo $9$, up to permuting
	$(a_0,a_1,a_2)$, multiplying some of the $a_i$ by $-1$ or multiplying 
	$(a_0,a_1,a_2)$ by an element of $(\ZZ/9\ZZ)^*$ (this gives a choice of $48$ elements
	of $(\ZZ/9\ZZ)^4$). 
	%Note: multiplying (1,4,7) by a unit is the same as permuting coordinates, hence why we get less than expected.
\end{remark}

\subsubsection{Volumes in the general case} 
To calculate the density $\sigma_p$, we shall integrate over the equivalence
classes of $\sim$. Namely, it follows from the above calculations that

\begin{align} \label{eqn:3_sigma_p}
	\sigma_p &= \sum_{i=1}^5\mu_p(\{ \mathbf{a}\in \ZZ_p^4 : 
	v(\mathbf{a}) \sim \boldsymbol{\delta}_i, S_{\mathbf{a}}(\QQ_p) \neq \emptyset\}) 
	=\sum_{i=1}^5 \mathcal{A}_i V_i,
\end{align} 
where 
$
	V_i=\mu_p(\{ \mathbf{a}\in \ZZ_p^4 : v(\mathbf{a}) \sim \boldsymbol{\delta}_i\}).
$
(Note that we have the relation $\sum_{i=1}^5V_i = 1$,
which can be used as a check for the calculations.)
\begin{lemma} \label{lem:3_f_i}
	For each $i=1,\ldots,5$, we have
	$$V_i = \frac{(1 - 1/p)^4}{(1-1/p^3)^4}f_i(1/p),$$
	where
	\begin{align*}
		f_1(x) & =  1 + x^{4} + x^{8}, 	\quad f_2(x) =  4( x + x^{5} + x^{6}), \quad 
		f_3(x) =  4( x^{2} +  x^{3} +  x^{7}),  \\
		f_4(x) &=  6( x^{2} +  x^{4} + x^{6}),\quad f_5(x) =  12( x^{3} + x^{4} + x^{5}).
	\end{align*}
\end{lemma}
\begin{proof}
	We first need to calculate all those elements of 
	$\ZZ_{\geq 0}^4$ which are equivalent to $\boldsymbol{\delta}_i$.
	We begin by  enumerating the set
	$C_i= \{\boldsymbol \alpha \in \ZZ_{\geq 0}^4: \boldsymbol \alpha \sim \boldsymbol{\delta}_i,
	w(\boldsymbol \alpha) \leq 8\}.$
	The value of $V_i$ is then the volume of the set
	of elements whose valuations lie in the coset
	$C_i + (3\ZZ_{\geq 0})^4.$ The corresponding volumes can then be calculated
	using the fact that 
	$$\mu_p(\{a \in \ZZ_p: 3 \mid v(a) \}) = \frac{1 - 1/p}{1-1/p^3}.$$
	For example, one has $C_1=\{(0,0,0,0),(1,1,1,1),(2,2,2,2)\}.$
	This implies that 
	\begin{align*}
		V_1 &=\mu_p(\{\mathbf{a} \in \ZZ_p^4:v(\mathbf{a}) \equiv \mathbf{0} \bmod 3\})
		\left( 1 + \frac{1}{p^4} + \frac{1}{p^8}\right)
=\frac{(1 - 1/p)^4}{(1-1/p^3)^4}\left( 1 + \frac{1}{p^4} + \frac{1}{p^8}\right),
	\end{align*}
	which proves the result in this case.	
	The calculations in the other cases proceed in a similar manner, 
	and the details are omitted.
\end{proof}

%Note that one easily checks that \eqref{eqn:3_sum_V_i} holds, as expected. 
Combining Lemma~\ref{lem:3_A_i} with \eqref{eqn:3_sigma_p} and
Lemma \ref{lem:3_f_i}, one 
easily confirms the values of $\sigma_p$ recorded in Theorem~\ref{thm:dog-handler} for $p\equiv 1\bmod{3}$ and $p=3$. \qed
%\begin{align*}
%\sigma_p & = \frac{(1 - 1/p)^4}{(1-1/p^3)^4}
%\left(1 + \frac{1}{p} + \frac{1}{p^2}\right)
%\left(1 - \frac{1}{p} + \frac{1}{p^2}\right)
%\left(1 + \frac{1}{p} + \frac{1}{3 p^2}\right)
%\left(1 + \frac{3}{p} + \frac{3}{p^2}\right) \\
%& = \frac{(1 - 1/p)^3}{(1-1/p^3)^3}
%\left(1 - \frac{1}{p} + \frac{1}{p^2}\right)
%\left(1 + \frac{1}{p} + \frac{1}{3 p^2}\right)
%\left(1 + \frac{3}{p} + \frac{3}{p^2}\right)
%\end{align*}
%for $p \equiv 1 \bmod 3$, and 
%\begin{align*}
%\sigma_p & = \frac{(1 - 1/p)^4}{(1-1/p^3)^4}
%\left(1 + \frac{1}{p} + \frac{1}{p^2}\right)
%\left(1 + \frac{3}{p} + \frac{46}{9p^2} + \frac{7}{p^3}
%	 + \frac{62}{9p^4} + \frac{19}{9p^5} + \frac{1}{p^6}\right)\\
%& = \frac{(1 - 1/p)^3}{(1-1/p^3)^3}
%\left(1 + \frac{3}{p} + \frac{46}{9p^2} + \frac{7}{p^3}
%	 + \frac{62}{9p^4} + \frac{19}{9p^5} + \frac{1}{p^6}\right)
%\end{align*}
%for $p=3$.
%T [don't need to write this out here]

\subsection{Diagonal quartic surfaces}

Consider the smooth biprojective hypersurface
$$
X=\{  a_0x_0^4+a_1x_1^4+a_2x_2^4+a_3x_3^4=0 \} \subset \PP^3\times \PP^3,
$$
over a number field $k$, viewed as the family of diagonal quartic  surfaces $X_P\subset \PP^3$.  
The hypotheses of Theorem \ref{thm:Ekedahl} are obviously met. Similarly, the same argument given for cubics implies that the conditions of Theorem~\ref{thm:H1} are all met, except possibly for Condition~\ref{hyp4}. However, if $S$ is a diagonal quartic surface over a field $L$ of characteristic zero, then a complete list for the possible choices for $H^1(L, \Pic S_{\bar L})$, depending on the Galois action of the lines of the surface, was compiled by Bright in \cite[\S A]{bright_thesis}. An inspection of this list reveals that for any number field $k$, the generic diagonal quartic over the function field $K$ of $\PP^3_k$ satisfies
$\H^1(K,\Pic  X_{\bar \eta}) \cong \ZZ/2\ZZ$; hence Condition~\ref{hyp4} also holds, as required.

Although we will not repeat the argument here, it is possible to calculate the exact proportion
$\sigma_\infty \prod_p \sigma_p$ of diagonal quartic surfaces over $\QQ$ which are everywhere locally soluble,
using a similar method to the proof of Theorem~\ref{thm:cubics}. There a few differences
however from the case of diagonal cubic surfaces. Firstly, one has $\sigma_p < 1$ for all primes $p$.
This is due to the fact that the surface
$$x_0^4 + px_1^4 + p^2x_2^4 + p^3x_3^4 = 0$$
over $\QQ_p$ never has a $\QQ_p$-point.
Secondly, one finds that the proportion is rather  smaller than the one occurring for diagonal cubic surfaces. 
 This comes from the fact that for diagonal quartics,
asking that the surface be soluble at the infinite place and at small primes imposes very 
strong conditions. In fact $\sigma_\infty = 3/4$, $\sigma_2 \approx 0.55$, $\sigma_3 \approx 0.87$ and 
$\sigma_5 \approx 0.79$, as can be confirmed numerically on a computer. 
This leads to the conclusion that  about  $28\%$ of diagonal quartic surfaces 
are soluble at infinity and at these primes.
In truth approximately 24\% of diagonal quartic surfaces over $\QQ$, when ordered by height, are everywhere locally soluble.
This leads to the following analogue of  Theorem \ref{thm:cubics}.

\begin{theorem}
Approximately $24\%$ of all diagonal quartics surfaces over $\QQ$, when ordered by height,  fail weak approximation. 
\end{theorem}

\section{The sieve of Ekedahl}\label{s:Ekedahl}

The aim of this section is to prove Theorem \ref{thm:Ekedahl}.
The main tool for this is the sieve of Ekedahl. This sieve
was first introduced by Ekedahl  \cite{Eke91}, but was developed extensively 
by Poonen--Stoll \cite{PS99b}
and  Bhargava \cite{Bha14}.  We shall push the analysis  further, by producing versions of this sieve
over general number fields, both in the context of integral points on affine space and 
rational points on projective space.

\subsection{Integral points} \label{sec:integral_points}

Let $\Lambda$ be a free $\ZZ$-module of finite rank and put 
$$
\Lambda_\infty = \Lambda \otimes_\ZZ \RR \quad \text{ and } \quad \Lambda_p = \Lambda \otimes_\ZZ \ZZ_p,
$$
for each prime $p$.
Choose Haar measures $\mu_\infty$ and $\mu_p$ on $\Lambda_\infty$ and $\Lambda_{p}$, respectively,
	such that $\mu_p(\Lambda_{p}) = 1$ for almost all primes $p$.
Our first result generalises work of Poonen and Stoll
\cite[Lem.~20]{PS99b} to deal with points lying in arbitrary free $\ZZ$-modules of finite rank  that are constrained to lie in  dilations of bounded subsets of $\Lambda_\infty$. For a subset $\Omega \subset X$ of a topological space $X$, we denote by $\partial \Omega = \overline{\Omega} \setminus \Omega^\circ$ its boundary.

\begin{lemma} \label{lem:sieve_of_Ekedahl}
	Let $\Lambda$ be a free $\ZZ$-module of finite rank $n$, let 
 $\Omega_\infty \subset \Lambda_\infty$ be a bounded subset and 
for each prime $p$ let $\Omega_p \subset \Lambda_{p}$.
	Assume that
	\begin{enumerate}
		\item $\mu_\infty(\partial \Omega_\infty) = 0$ and $\mu_p(\partial \Omega_p) = 0$
		for each prime $p$,
		\item $\mu_\infty(\Omega_\infty) > 0$ and $\mu_p(\Omega_p) > 0$ for each prime $p$.
	\end{enumerate}	
	Suppose also that 
	\begin{equation} \label{eqn:lim_M}
	\lim_{M \to \infty} \limsup_{B \to \infty}
	\frac{\# \{\x \in \Lambda\cap  B\Omega_\infty: \exists~\text{a prime $p > M$
	 s.t.  } \x  \not \in \Omega_p\}}{B^n
	 } = 0.
	\end{equation}
	Then the limit
	$$ \lim_{B \to \infty} \frac{1}{B^n}
	\# \{\x \in \Lambda\cap B\Omega_\infty:  \x \in \Omega_p \text{ for all primes $p$}\}$$
	exists, is non-zero and equals
	$$\frac{\mu_\infty(\Omega_\infty)}{\mu_\infty(\Lambda_\infty/\Lambda)}
	\prod_p \frac{\mu_p(\Omega_p)}{\mu_p(\Lambda_p)}.$$
\end{lemma}
\begin{proof}
	Note that (1) implies that the sets $\Omega_\infty$ 
	and $\Omega_p$ are 
	Jordan measurable,
	hence measurable.
	We equip $\RR^n$ and $\ZZ_p^n$ with the usual Haar measures and choose
	an isomorphism $\Lambda \cong \ZZ^n$. Then the measures $\mu_\infty$
	and $\mu_p$ induce measures on $\RR^n$ and $\ZZ_p^n$ which differ from the usual
	Haar measures by $\mu_\infty(\Lambda_\infty/\Lambda)$ and $\mu_p(\Lambda_p)$, respectively.
	Hence it suffices to prove the result when $\Lambda = \ZZ^n$ and $\mu_\infty$ and $\mu_p$
	are the usual Haar measures.
	
	Our argument closely follows the proof of \cite[Lem.~20]{PS99b} (although we shall be working with the complements of the sets $U_\infty,U_p$ considered there). We begin by dealing with the case that there exists $M$ such that $\Omega_p=\ZZ_p^n$ for all finite  $p>M$.  For any prime $p$ a {\em box} $K_p\subset \ZZ_p^n$ is defined to be a 
 cartesian product of closed balls of the shape $\{x\in \ZZ_p: |x-a|_p\leq b\}$, for $a\in \ZZ_p$ and $b\in \RR$.
Let $P=\prod_{p\leq M}\Omega_p$ and $Q=\prod_{p\leq M}(\ZZ_p^n\setminus \Omega_p)$.
By hypothesis  $\Omega_p$ and $\ZZ_p^n\setminus \Omega_p$ have boundary of measure zero.
Hence by compactness we can cover the closure $\overline{P}$ of $P$ (resp.~the closure $\overline{Q}$ of $Q$) by a finite number of boxes $\prod_{p\leq M} I_p$ (resp.~$\prod_{p\leq M} J_p$) the sum of whose measures is arbitrarily close to the measure of $\overline{P}$ (resp.~$\overline Q$), which equals $\prod_{p} \mu_p(\Omega_p)$ (resp.~$1-\prod_p \mu_p(\Omega_p)$).
We claim that 
	$$ \lim_{B \to \infty}
\frac{1}{B^{n}}\#\left\{\x\in \ZZ^n\cap B\Omega_\infty:  \x \in \prod_{p\leq M} K_p\right\} = \mu_\infty(\Omega_\infty)\prod_{p}
	\mu_p(K_p),
	$$
	for any box  $\prod_p K_p$.
	Indeed, the set of those $\x\in \ZZ^n$ for which $\x\in \prod_{p\leq M}K_p$ 
	is a translate of a sublattice, the determinant of which is $\prod_{p\leq M}\mu_p(K_p)^{-1}$. 
The claim therefore follows from a classical lattice point counting result,
which applies here as $\Omega_\infty$ is Jordan measurable (see Lemma 2 and its proof in  \cite[\S 6]{marcus}, for example). Applying this with $\prod_{p\leq M} K_p$ taken to be 
first $\prod_{p\leq M} I_p$ and second $\prod_{p\leq M} J_p$, we easily complete the proof of the lemma when 
 there exists $M$ such that $\Omega_p=\ZZ_p^n$ for all finite  $p>M$.  
 
We now turn to the general case. For $M \leq M'\leq \infty$ 
and $B > 0$, let
$$f_{M,M'}(B) = \frac{1}{B^n} \#\{\x \in \ZZ^n\cap B\Omega_\infty:  \x \in \Omega_p \text{ for all primes }p \in [M,M')\}$$
and put  $f_{M}(B)=f_{1,M}(B)$. 
Note that $f_M(B) \geq f_{M+r}(B)$ for any $r \in \NN \cup \{\infty\}$.
The hypothesis \eqref{eqn:lim_M} implies that
\begin{equation}	\label{eqn:uniform}
	\lim_{M \to \infty} \limsup_{B \to \infty} (f_M(B) - f_{\infty}(B)) = 0.
\end{equation}
Moreover, our work so far shows that
\begin{equation} \label{eqn:pointwise}
	\lim_{B \to \infty} f_{M,M'}(B) = \mu_\infty (\Omega_\infty)\prod_{M\leq p < M'} \mu_p(\Omega_p), \quad \text{for all $M<M'<\infty$}.
\end{equation}
Thus \eqref{eqn:uniform} and \eqref{eqn:pointwise} together imply that
$$\lim_{B \to \infty} f_\infty(B)
= \lim_{M \to \infty} \lim_{B \to \infty} f_M(B) = 
\mu_\infty (\Omega_\infty) \lim_{M \to \infty} \prod_{p < M} \mu_p(\Omega_p).$$
To complete the proof, it suffices to show the convergence of the above infinite product.
But \eqref{eqn:lim_M} and \eqref{eqn:pointwise} together imply that
$$
\lim_{M \to \infty} \sup_{r \in \NN} \left|1 - \prod_{M \leq p < M + r} \mu_p(\Omega_p)\right| 
= \frac{1}{\mu_\infty(\Omega_\infty)}\lim_{M \to \infty} \sup_{r \in \NN} \lim_{B\to \infty} 
\left|f_{1}(B) - f_{M,M+r}(B)\right| = 0.
$$
Thus the infinite product converges by Cauchy's criterion, as required.
\end{proof}

We now use Lemma \ref{lem:sieve_of_Ekedahl} to obtain a version over number fields. The following notation will remain in use for the rest of this section. 	Let $k$ be a number field of degree $d$ over $\QQ$ with ring of integers $\mathfrak{o}$ and discriminant $\Delta_{k}$.
We identify  $\mathfrak{o}$ with its image  as a rank $d$ lattice inside the commutative $\RR$-algebra  $k_\infty=\mathfrak{o}\otimes_\ZZ \RR$. It has covolume 
$|\Delta_{k}|^{1/2}$ with respect to the usual Haar measure $\mu_\infty$ (see e.g.~\cite[Prop.~I.5.2]{Neu99}).
For any rational prime $p$ we
have 
$$
\mathfrak{o}\otimes_\ZZ \ZZ_p=\prod_{\mathfrak{p}\mid p} \mathfrak{o}_{\fp}
$$
where
$\mathfrak{o}_{\fp}$ is the ring of integers of the completion $k_\fp$ of $k$ at a prime ideal $\fp$.
We equip each $\mathfrak{o}_{\fp}$ with the Haar measure $\mu_\fp$ normalised so that 
$\mu_\fp(\mathfrak{o}_{\fp}) = 1$.
We put
$\FF_\mathfrak{p}=\fo/\fp$ for  the residue field and $\Norm \mathfrak{a}$ for the ideal norm of
any fractional ideal $\mathfrak{a}$ of $k$. We equip $k_\infty^n$ and each $\fo_\fp^n$
with the induced product measures, which by abuse of notation we also denote by
$\mu_\infty$ and $\mu_\fp$, respectively.
The following result is now an easy consequence of 
Lemma~\ref{lem:sieve_of_Ekedahl}.

\begin{proposition} \label{prop:sieve_of_Ekedahl_affine}
Let $k/\QQ$ be a number field of degree $d$, with  notation as above.
Let $\Omega_\infty \subset k_\infty^n$ be a bounded subset and let $\Omega_\mathfrak{p}\subset\mathfrak{o}_\fp^n$
for each prime  $\mathfrak{p}$.  Assume that
	\begin{enumerate}
		\item$ \mu_\infty(\partial \Omega_\infty) = 0$ and $\mu_\fp(\partial \Omega_\fp) = 0$ 
		for each prime  $\fp$,
		\item $\mu_\infty(\Omega_\infty) > 0$ and $\mu_\fp(\Omega_\fp) > 0$ 
		for each prime   $\fp$.
	\end{enumerate}	
	Suppose also that 
	\begin{equation} \label{eqn:lim_M_affine}
	\lim_{M \to \infty} \limsup_{B \to \infty}
	\frac{\# \{\x \in \mathfrak{o}^n\cap  B\Omega_\infty: \exists~\textrm{$\fp$ s.t.  $\Norm \fp > M$ and } 
	 \x  \not \in \Omega_\fp\}}{B^{dn}} = 0.
	\end{equation}
	Then the limit
	$$ \lim_{B \to \infty} \frac{1}{B^{dn}}
	\# \{\x \in \mathfrak{o}^n\cap  B\Omega_\infty: \x \in \Omega_\mathfrak{p} \textrm{ for all primes } \mathfrak{p}\}$$
	exists, is non-zero and equals
	$$\frac{1}{|\Delta_{k}|^{n/2}}\mu_\infty(\Omega_\infty)\prod_{\mathfrak{p}} \mu_\mathfrak{p}(\Omega_\mathfrak{p}).$$
\end{proposition}

The following result furnishes us with  a large class of subsets $\Omega_\fp$ which satisfy \eqref{eqn:lim_M_affine}.

\begin{lemma} \label{lem:codim_2_affine}
	Let $k/\QQ$ be a number field and let $\mathcal{Z} \subset \mathbb{A}^n_{\mathfrak{o}}$
	be a closed subset of codimension at least two defined over $\fo$. Let $\Omega_\infty \subset k_\infty^n$
	be a bounded subset with $\mu_\infty(\partial \Omega_\infty) = 0$ and $\mu_\infty(\Omega_\infty) > 0$. Let 
	$$\Omega_\mathfrak{p} = \{\x \in \mathfrak{o}_{\mathfrak{p}}^n: \x \bmod \mathfrak{p} 
	\not \in \mathcal{Z}(\FF_\mathfrak{p}) \},$$
for each prime ideal $\fp$.
	Then  \eqref{eqn:lim_M_affine} holds.
\end{lemma}
\begin{proof}
	When $k=\QQ$ this is due to Ekedahl \cite[Thm.~1.2]{Eke91}. (See also \cite[Thm.~3.3]{Bha14} for a version
	with an effective error term.) This
	method generalises to give the
	result over any number field \cite[Thm.~18]{BSW15}.
\end{proof}

%I'm not sure if this example should make it into the final version.
%I put it in to put my mind at ease.
%T [I thought it was worth making the remark without going through the details.]

Note that  the conclusion of Lemma \ref{lem:codim_2_affine} is generally  false for  subsets of codimension one, as 
consideration of the subset cut out by a single linear equation in $\mathbb{A}_\fo^n$ readily confirms.

%\begin{example}
%	Note that this result fails for subsets of codimension one in general.
%	For example, consider the closed subset $\mathcal{Z} = \{ x_0 = 0\} \subset \PP^n.$
%	Here one is interested in the limit
%	$$\lim_{M \to \infty} \rho(\{\mathbf{x} \in \ZZ^n: \textrm{There exists a prime } p > M
%	\textrm{ such that } p \mid x_0\}).$$
%	I claim that this limit exists and is equal to $1$. To show this,
%	it suffices to consider
%	$$\lim_{M \to \infty}\lim_{B \to \infty}
%	\frac{1}{B}\{ x \in \NN: x \leq B, \textrm{there exists a prime } p > M
%	\textrm{ such that } p \mid x\}.$$
%	Note that the numbers being counted here are exactly the non $M$-smooth numbers.
%	Let $\Psi(M,B)$ denote the number of $M$-smooth numbers of size less than $B$.
%	We have the uniform upper bound
%	$$\Psi(M,B) \ll B \exp(-1/2 (\log B) / (\log M) ),$$
%	which in particular implies that the density of $M$-smooth numbers
%	is always $0$, from where the claim follows.
%	This is consistent with the fact that 
%	$$\rho(\{\mathbf{x} \in \ZZ^n: p \nmid x_0 \text{ for all primes } p\}) = 0.$$
%\end{example}

\subsection{Rational points}\label{s:goat-herder}
We now combine the affine version of the sieve of Ekedahl with the method
of Schanuel \cite{Sch79} to obtain a version of this sieve for rational points in projective space over number fields $k/\QQ$.
This gives a version of Schanuel's theorem in which one is  allowed to impose infinitely many
local conditions. 
%T% [removed the following; didn't thing it added much]
%The sieve of Ekedahl also allows us to give a slightly shorter proof of Schanuel's theorem,
%since it allows us to bypass the need to perform a M\"{o}bius inversion. This means
%that we need zero uniformity in our error terms.

For any field $F$ and any subset $\Omega \subset \PP^n(F)$, we denote by
$\Omega^{\mathrm{aff}}$ the affine cone of $\Omega$. This is the pull-back of $\Omega$
via the map
$\mathbb{A}^{n+1}(F) \setminus \{0\} \to \PP^n(F).$
We denote by 
$$H_\nu: k_\nu^{n+1} \to \RR_{\geq 0}, \quad (x_0,\ldots, x_n) 
\mapsto \max\{\|x_0\|_\nu, \ldots, \|x_n\|_\nu \}.$$
where $\|\cdot\|_\nu=|\cdot|_\nu^{[k_\nu:\QQ_\lambda]}$ and $\lambda$ is unique place of $\QQ$ lying below
$\nu$.
The product of these local height functions $\prod_{\nu \in \Val(k)}H_\nu$
descends to a well-defined height function   $H:\PP^n(k)\to \RR_{\geq 1}$,
which satisfies
$$
H(x)=\frac{1}{\Norm(\langle x_0,\dots,x_n\rangle)}\prod_{\nu \mid \infty}H_\nu(x),
$$
for any choice of representative $x=(x_0:\dots:x_n)\in \PP^n(k)$. Here  $\langle x_0,\dots,x_n\rangle$ denotes the $\fo$-span of $x_0,\dots,x_n$.  
Bearing this notation in mind we proceed by establishing the following result.

\begin{proposition} \label{prop:sieve_of_Ekedahl_proj}
	Let $k/\QQ$ be a number field of degree $d$. For each  $\nu\in \mathrm{Val}(k)$ 
	let $\Omega_\nu \subset \PP^n(k_\nu)$ be a subset such that 
 $\mu_\nu(\partial \Omega_\nu^\mathrm{aff}) = 0$  and 
 $\mu_\nu(\Omega_\nu^\mathrm{aff}) > 0$. 
 	Suppose that
\begin{equation}\label{eq:goat}
	\lim_{M \to \infty} \limsup_{B \to \infty}
	\frac{\# \{\x \in \mathfrak{o}^{n+1}\cap B\Psi  : 	\exists~\fp \text{ s.t. }\Norm\mathfrak{p} > M \text{ and }
	\x  \not \in \Omega_\mathfrak{p}^{\mathrm{aff}}\}}{B^{d(n+1)}} = 0,
\end{equation}
for all bounded subsets $\Psi\subset k_\infty^{n+1}$ of positive measure
with $\mu_\infty(\partial \Psi) = 0$.
	Then the limit
	$$ \lim_{B \to \infty} \frac{\# \{x \in \PP^n(k): H(x) \leq B, x \in \Omega_\nu \mbox{ for all } \nu\in \mathrm{Val}(k)\}}
	{\# \{x \in \PP^n(k): H(x) \leq B\}},$$
	exists, is non-zero and equals
	$$\prod_{\nu \mid \infty} \frac{\mu_\nu(\{\x \in \Omega_\nu^\mathrm{aff}: H_\nu(\x) \leq 1\})}
	{\mu_\nu(\{\x \in k_\nu^{n+1}: H_\nu(\x) \leq 1\})}	\prod_{\mathfrak{p}} 
	\mu_\fp(\{\x \in \Omega_\mathfrak{p}^\mathrm{aff} \cap\mathfrak{o}_\mathfrak{p}^{n+1}\}).$$
\end{proposition}
\begin{proof}
	Choose integral ideal representatives $\mathfrak{c}_1,\ldots, \mathfrak{c}_h$ of
	the elements of the class group of $k$. We obtain a partition
	$$\PP^n(k) = \bigsqcup_{i=1}^h \{ (x_0:\cdots:x_n) \in \PP^n(k) : [x_0,\ldots,x_n] = [\mathfrak{c}_i]\}.$$
	Here $[x_0,\ldots,x_n]$ denotes the element of the class group of $k$ given by the fractional
	ideal $\langle x_0,\ldots,x_n\rangle $. Let $\mathfrak{c}\in \{\mathfrak{c}_1,\dots,\mathfrak{c}_h\}$ and put 
 $\Omega = (\Omega_\nu)_{\nu \in \Val(k)}$. We are interested in the counting function
	$$N(\Omega, \mathfrak{c}, B) = \# \{x \in \PP^n(k): H(x) \leq B, ~[x] = [\mathfrak{c}], ~
	x \in \Omega_\nu ~\forall~\nu \in \Val(k)\}.$$
Let $\mathfrak{F} \subset k_\infty^{n+1}$ be the fundamental domain for the action of the units
	of $k$ constructed by Schanuel  \cite[\S 1]{Sch79}.
As before we identify $\mathfrak{o}^{n+1}$ with its image in $k_\infty^{n+1}$ as a sublattice of full rank. Then we have
	$$
	N(\Omega, \mathfrak{c}, B) = \frac{1}{w_k}\# \left\{\x \in \mathfrak{F} \cap \mathfrak{o}^{n+1}:
\begin{array}{l}
	 H_\infty(\x) \leq B(\Norm \mathfrak{c}),
	  \langle \x \rangle = \mathfrak{c}, \\
	  \x \in \Omega_\nu^\mathrm{aff} ~\forall~ \nu \in \Val(k)
	  \end{array}\right\},
	  $$
  where $w_k$ denotes the number of roots of unity in $\fo$ and 
  $H_\infty = \prod_{\nu \mid \infty}H_\nu$.
  
The condition
	 $\langle \x \rangle = \mathfrak{c}$ is in fact a collection of local conditions, being  equivalent to asking that
	$ \langle \x \rangle_{\mathfrak{p}} = \mathfrak{c}_{\mathfrak{p}}$
	for all prime ideals $\mathfrak{p}$ of $k$, where 
 for an integral   ideal $\mathfrak{a}$  we write $\mathfrak{a}_{\mathfrak{p}} = 
	\mathfrak{a} \otimes_{\mathfrak{o}} \mathfrak{o}_{\mathfrak{p}}.$
	Hence 
	$$
		N(\Omega, \mathfrak{c}, B) = \frac{1}{w_k}\# \left\{\x \in  \Omega_\infty^\mathrm{aff}\cap 
		\mathfrak{F} \cap \mathfrak{o}^{n+1}:
\begin{array}{l}
	 H_\infty(\x) \leq B(\Norm\mathfrak{c}),\\
	 \langle \x \rangle_{\mathfrak{p}} = \mathfrak{c}_{\mathfrak{p}} \text{ and } \x \in \Omega_\mathfrak{p}^\mathrm{aff}
	 ~\text{$\forall$ primes } \mathfrak{p}
	  \end{array}\right\},
$$
where $\Omega_\infty = \prod_{\nu \mid \infty} \Omega_\nu$.
By \cite[Prop.~2]{Sch79}, the set $\mathfrak{F}(1) = \mathfrak{F} \cap \{\x \in k_\infty^{n+1} : H_\infty(\x)\leq 1\}$
is bounded and has  Lipschitz parametrisable boundary, hence is Jordan measurable (see \cite{Spa95}).
As the intersection of two Jordan measurable sets is Jordan measurable, we find that 
$$
\Theta_\infty= \{\x\in\Omega_\infty^\mathrm{aff} \cap \mathfrak{F}: H_\infty(\x)\leq \Norm\mathfrak{c}\}$$ 
is bounded and Jordan measurable, hence satisfies the conditions of 
Proposition \ref{prop:sieve_of_Ekedahl_affine}. Next, let 
$\Theta_\mathfrak{p} = \{\x \in \Omega_\mathfrak{p}^\mathrm{aff} \cap \mathfrak{o}_\fp^{n+1} : \langle \x \rangle_{\mathfrak{p}} 	= \mathfrak{c}_{\mathfrak{p}}\}$. The Jordan measurability of each $\Theta_\fp$
is clear, being the intersection of two Jordan measurable sets. We also claim that the $\Theta_\fp$ satisfy \eqref{eqn:lim_M_affine}.
	Indeed, given our assumption \eqref{eq:goat}, an application of De Morgan's laws shows that it suffices
	to note that the sets
	$\{\x \in \mathfrak{o}_\mathfrak{p}^{n+1} : \langle \x \rangle_{\mathfrak{p}} 
	= \mathfrak{o}_{\mathfrak{p}}\}$
	satisfy \eqref{eqn:lim_M_affine} (this follows, for example, from applying Lemma \ref{lem:codim_2_affine}
	to the subscheme $x_0 = \cdots = x_n=0$).
	We are therefore  in a position to apply Proposition~\ref{prop:sieve_of_Ekedahl_affine} to deduce that
	$$
	\lim_{B \to \infty}\frac{N(\Omega, \mathfrak{c}, B) }{B^{n+1}} 
	=\frac{1}{w_k|\Delta_{k}|^{(n+1)/2}}\mu_\infty(\Theta_\infty)
	\prod_{\mathfrak{p}} \mu_\mathfrak{p}(\Theta_\mathfrak{p}).$$
	For the non-archimedean densities we have
	\begin{align*}
	\mu_\mathfrak{p}(\Theta_\mathfrak{p}) 
	&= \frac{1}{(\Norm_\mathfrak{p} \mathfrak{c})^{n+1}} 
	\mu_\mathfrak{p}(\{\x \in \Omega_\mathfrak{p}^\mathrm{aff} \cap \mathfrak{o}_\fp^{n+1} : \langle \x \rangle_{\mathfrak{p}} = \mathfrak{o}_{\mathfrak{p}}\}) \\
	&= \frac{1}{(\Norm_\mathfrak{p}\mathfrak{c})^{n+1}}\left(1 - \frac{1}{(\Norm_\mathfrak{p}\mathfrak{p})^{n+1}}\right)
	\mu_\mathfrak{p}(\{\x \in \Omega_\mathfrak{p}^\mathrm{aff} \cap \mathfrak{o}_\fp^{n+1}\}).
	\end{align*}
	The  archimedean density is 
	\begin{align*}
	\mu_\infty(\Theta_\infty) = (\Norm \mathfrak{c})^{n+1}
	\mu_\infty(\{\x \in \Omega_\infty^\mathrm{aff} \cap \mathfrak{F} : H_\infty(\x) \leq 1\}).
	\end{align*}
In the usual way (see  \cite[p.443]{Sch79} or \cite[Lem.~5.1]{FP13}, for example), we obtain
	$$\frac{\mu_\infty(\{\x \in \Omega_\infty^\mathrm{aff} \cap \mathfrak{F}: H_\infty(\x) \leq 1\})}
	{\mu_\infty(\{\x \in\mathfrak{F}  : H_\infty(\x) \leq 1\})} 
	= \prod_{\nu \mid \infty} \frac{\mu_\nu(\{\x \in \Omega_\nu^\mathrm{aff} : H_\nu(\x) \leq 1\})}
	{\mu_\nu(\{\x \in k_\nu^{n+1} : H_\nu(\x) \leq 1\})},$$
which proves the result.
\end{proof}

On passing to the affine cone, the following is a straightforward consequence of  Lemma~\ref{lem:codim_2_affine}.

\begin{lemma} \label{lem:codim_2_proj}
	Let $k$ be a number field and let $\mathcal{Z} \subset \PP^n_\fo$
	be a closed subset of codimension at least two.
	Let 
	$$\Omega_\mathfrak{p} = \{\x \in \PP^n(\mathfrak{o}_{\mathfrak{p}}): \x \bmod \mathfrak{p} 
	\not \in \mathcal{Z}(\FF_\mathfrak{p}) \},$$
	for each prime ideal. Then  \eqref{eq:goat} holds for any bounded
	subset $\Psi\subset k_\infty^{n+1}$ with positive measure and
	$\mu_\infty(\partial \Psi) = 0$.
\end{lemma}

%Is $$\PP^n(\mathfrak{o}_{k,\mathfrak{p}}) \setminus \mathcal{Z}(\mathfrak{o}_{k,\mathfrak{p}}) 
%= (\PP^n \setminus \mathcal{Z})(\mathfrak{o}_{k,\mathfrak{p}})$$
%?
%No! Consider $\infty \subset \PP^1$.
%In fact, I really ? need to consider reduction modulo $\mathfrak{p}$..?

\subsection{Proof of Theorems \ref{thm:Ekedahl} and \ref{thm:Ekedahl_affine}}
Let $k$ be a number field with ring of integers $\fo$.
We begin with the following result, which generalises work of Skorobogatov
\cite[Lem.~2.3]{Sko96}.
The statement is a little more general than we will require in this section, because we will use it again in Section~\ref{s:bull-rider}.  Recall that two varieties $X,Y$ over a field $F$ are said to be \emph{geometrically isomorphic} if they become isomorphic after base change to an algebraic closure of $F$.

%T% [Nicer just to record the more general lemma]

%\begin{lemma}\label{lem:split_surjective}
%	Let $\pi:X \to Y$ be a surjective morphism of varieties over $k$.
%	Assume that the fibre over every point of $Y$ is split. Then
%	there exist a non-empty open subset $U \subset \Spec \mathfrak{o}_k$, together
%	with models $\mathcal{X}$ and $\mathcal{Y}$ of $X$ and $Y$
%	over $U$, such that for all $\mathfrak{p} \in U$ the map
%	$$\mathcal{X}(\mathfrak{o}_{k,\mathfrak{p}}) \to \mathcal{Y}(\mathfrak{o}_{k,\mathfrak{p}}),$$
%	is surjective.
%\end{lemma}
%\begin{proof}	
%	Choose a non-empty open subset $U \subset \Spec \mathfrak{o}_{k}$
%	together with models $\mathcal{X}$ and $\mathcal{Y}$
%	for $X$ and $Y$ over $U$. As splitness is a constructable property
%	in the sense of EGA ?, we may shrink $U$ if necessary, to assume that 
%	the fibre over every point of $\mathcal{Y}$ is split. The Lang-Weil estimate
%	now imply to that, shrinking $U$ again if necessary, for all $\mathfrak{p}  \in U$
%	the fibre over every element of $\mathcal{Y}(\FF_\mathfrak{p})$ 
%	contains a non-singular $\FF_\mathfrak{p}$-point.
%	The result now follows from Hensel's Lemma.
%\end{proof}

\begin{lemma}\label{lem:lw}
Let $f \colon \sX \to \sY$ be a morphism of integral schemes, both separated and of finite type over $\fo$.  Suppose that the generic fibre of $f$ is split.  Then there exist non-empty open subsets $U \subset \Spec \fo$ and $\sV \subset \sY$ such that, for all $\fp \in U$, 
and for all $\bar{y} \in \sV(\overline{\FF}_\fp)$, any $\FF_\fp$-variety geometrically isomorphic to $\sX_{\bar{y}}$ has a $\FF_\fp$-rational point.  (In particular, for all $y \in \sV(\FF_\fp)$, the fibre $\sX_y$ has an $\FF_\fp$-rational point.)

If in addition the fibre of $f$ over every point of codimension one in $\sY$ is split, then we can take $\sV$ to have complement of codimension at least two in $\sY$.
\end{lemma}

\begin{proof}
As the generic fibre of $f$ is split, there is a dense open set $\sV \subset \sY$ above which all fibres are split.  If the fibre of $f$ at every point of codimension one is also split, then we can take $\sV$ to have complement of codimension at least two.  Removing a closed subset of $\sX$ makes the fibres of $\sX$ above $\sV$ geometrically integral.
In particular, the fibres are generically smooth; removing a further closed subset of $\sX$ makes the fibres smooth and geometrically irreducible.
%The property that the fibre of $f$ above a point $y \in \sY$ be split is a constructible property (since the same is true for the property of being geometrically integral); so there is a dense open set $\sV \subset \sY$ above which all fibres are split.  If the fibre of $f$ at every point of codimension one is split, then we can take $\sV$ to have complement of codimension at least two.  Removing a closed subset of $\sX$ makes the fibres of $\sX$ above $\sV$ geometrically integral.

Pick an auxiliary prime $\ell$, and set $U = \Spec \fo \setminus \{\ell\}$.  The sheaves $\R^i f_! \QQ_\ell$ are constructible sheaves of $\QQ_\ell$-modules on $\sY$, and zero for $i$ sufficiently large.  It follows that the function sending a geometric point $\bar{y}$ of $\sY$ to the $i$th compactly supported Betti number $\dim \H^i_c(X_{\bar{y}},\QQ_\ell)$ is constructible.  Since $\sY$ is quasi-compact, these Betti numbers take only finitely many values.

Fix a point $\bar{y} \in \sV(\overline{\FF}_\fp)$, and let $Z$ be a variety over $\FF_\fp$ geometrically isomorphic to the fibre $\sX_{\bar{y}}$.  Because $\sX_{\bar{y}}$ is smooth and geometrically irreducible, Poincar\'e duality shows that its top Betti number is $1$.  The Lefschetz trace formula, together with Deligne's bound~\cite[Th\'eor\`eme~1]{deligne} on the traces of the Frobenius operators, shows that $Z$ has an $\FF_\fp$-point as long as $\FF_\fp$ is sufficiently large and $\fp$ is prime to $\ell$.  Shrinking $U$ to exclude all primes too small for this to apply gives the result.
\end{proof}

\begin{corollary}\label{cor:split_2_surjective}
	Let $\pi:X \to Y$ be a dominant morphism of varieties over $k$.
	Suppose that the fibre of $\pi$ over every point of 
	codimension one is split and that the generic fibre of $\pi$ is geometrically
	integral.
	%D: Does the first condition imply the second?
	%M: No.
	Then there exist a non-empty open subset $U \subset \Spec \mathfrak{o}$, together
	with models $\mathcal{X}$ and $\mathcal{Y}$ of $X$ and $Y$ over $U$
	and a  closed subset $\mathcal{Z} \subset \mathcal{Y}$ of codimension 
	at least two, such that 
	the map
	\[
	(\mathcal{X} \setminus \pi^{-1}(\mathcal{Z})) (\mathfrak{o}_{\mathfrak{p}}) \to 
	(\mathcal{Y} \setminus \mathcal{Z})(\mathfrak{o}_{\mathfrak{p}})
	\]
	is surjective
	for all $\mathfrak{p} \in U$.
\end{corollary}
\begin{proof}
Choose a model $f \colon \sX \to \sY$ of $\pi$ over a non-empty open subset $U$ of $\Spec \fo$.  Shrinking $U$, we can assume that the fibre of $f$ over each point of codimension one in $\sY$ is split.  Now using Lemma~\ref{lem:lw} on the smooth locus of $f$ and applying Hensel's Lemma gives the result.
\end{proof}

We now come to the main result of this section, which implies  Theorem~\ref{thm:Ekedahl}.

\begin{theorem} \label{thm:Ekedahl2}
	 Let $k$ be a number field.
	 Let $\pi:X\to \PP^n$ be a dominant quasi-projective $k$-morphism with geometrically integral generic fibre.
	 Assume that:
	 \begin{enumerate}
	 \item
	 the fibre of $\pi$ at each codimension $1$ point of $\PP^n$ is split,
	 \item
	$X(\Adele_k) \neq \emptyset$.
	\end{enumerate}
	Then the limit $\sigma(\pi)$ given in \eqref{eq:horse} 
	exists, is non-zero and equals
	$$\prod_{\nu \mid \infty} \frac{\mu_\nu(\{\x \in \pi(X(k_\nu))^\mathrm{aff}: H_\nu(\x) \leq 1\})}
	{\mu_\nu(\{\x \in k_\nu^{n+1}: H_\nu(\x) \leq 1\})}	
	\prod_{\fp} \mu_\fp(\{\x \in \pi(X(k_\fp))^\mathrm{aff} \cap\mathfrak{o}_\fp^{n+1}\}).
	$$
\end{theorem}

Specialising to $k=\QQ$, it is easy to see that the product of local densities becomes
\begin{equation} \label{eqn:euler_product}
	\frac{1}{2^{n+1}} \mu_\infty(\{\x \in \pi(X(\RR))^\mathrm{aff}: H_\infty(\x) \leq 1\})
	\prod_{p} \mu_p(\{\x \in \pi(X(\QQ_p))^\mathrm{aff} \cap\ZZ_p^{n+1}\}).
\end{equation}

To prove Theorem \ref{thm:Ekedahl2},  we let 
$\pi: X \to \PP^n$ be as in the statement.
We want to show that
Proposition \ref{prop:sieve_of_Ekedahl_proj} applies, with the choice
$
\Omega_\nu = \pi(X(k_\nu))
$
for each $\nu\in \mathrm{Val}(k)$.
Beginning with hypothesis \eqref{eq:goat}, it follows from  Corollary~\ref{cor:split_2_surjective} that 
there exists 
a closed subset $\mathscr{Z} \subset \PP_\fo^{n}$ 
of  codimension at least two, 
such that
$(\PP_{\fo_\fp}^n\setminus \mathscr{Z})(\mathfrak{o}_\mathfrak{p}) \subset \pi(X(k_\mathfrak{p}))$
for all sufficiently large primes ideals $\mathfrak{p}$.
Thus 
$$\{\x \in \PP^n(\mathfrak{o}_{\mathfrak{p}}): \x \bmod \mathfrak{p} \not \in \mathscr{Z}(\FF_\mathfrak{p}) \}
\subset \pi(X(k_\mathfrak{p}))$$
for all sufficiently large primes ideals $\mathfrak{p}$.
Condition \eqref{eq:goat} therefore follows from Lemma~\ref{lem:codim_2_proj}.
The following result, which we apply to the affine cone of $\pi$, establishes the
measurability conditions, and so concludes the  proof of Theorem \ref{thm:Ekedahl2}.
For a place $\nu \in \Val(k)$, we denote by $\mu_{\nu}$ the corresponding measure, as 
defined in \S \ref{sec:integral_points}.

\begin{lemma} \label{lem:Tarski}
	Let $\nu\in \mathrm{Val}(k)$ and let $X$ be a quasi-projective variety
	over $k_\nu$ equipped with a dominant $k_\nu$-morphism
	$\pi:X \to \mathbb{A}^n$. Assume that $X(k_\nu) \neq \emptyset$. Then $\pi(X(k_\nu))$
	is measurable with respect to $\mu_\nu$ and 
	$$\mu_\nu(\partial(\pi(X(k_\nu))))=0 \quad \text{ and } \quad \mu_\nu(\pi(X(k_\nu)))>0.$$
\end{lemma}
\begin{proof}
	The Tarski--Seidenberg--Macintyre theorem  implies
	that $\pi(X(k_\nu))$ is a semi-algebraic set (see \cite[Thm.~3]{Pon95}). From this, it easily
	follows that $\pi(X(k_\nu))$ is measurable and that its boundary has measure zero.
	It thus suffices to show that it has positive measure. To do this, we may replace
	$X$ by an open subset if necessary to assume that $\pi$ is smooth. 
	Then the induced map $X(k_\nu) \to \mathbb{A}^n(k_\nu)$ is a submersion,
	hence $\pi(X(k_\nu)) \subset \mathbb{A}^n(k_\nu)$ is open and thus has positive measure, as required.
\end{proof}

%T% [new proof of extra theorem]

We close this section by indicating the proof of Theorem \ref{thm:Ekedahl_affine}, which runs very similarly to the proof of 
Theorem \ref{thm:Ekedahl2}.
Let  $\pi: X \to \mathbb{A}^n$ and let $\Psi\subset k_\infty^n$  be as in Theorem \ref{thm:Ekedahl_affine}.
We want to show that 
Proposition \ref{prop:sieve_of_Ekedahl_affine} applies, with  $\Omega_\infty=\Psi$ and 
$
\Omega_\fp = \pi(X(k_\fp))
$
for each prime ideal $\fp$. Note that, by our assumptions on $\Psi$, we have $B\Psi \subset \pi(X(k_\infty))$ 
for all $B \geq 1$.
To check  hypothesis \eqref{eqn:lim_M_affine}, we first deduce from  Corollary~\ref{cor:split_2_surjective} that 
there exists 
a closed subset $\mathscr{Z} \subset \mathbb{A}_\fo^{n}$ 
of  codimension at least two, 
such that
%$(\mathbb{A}_{\fo_\fp}^n\setminus \mathscr{Z})(\mathfrak{o}_\mathfrak{p}) \subset \pi(X(k_\mathfrak{p}))$
%for all sufficiently large primes ideals $\mathfrak{p}$.
%Thus 
$$\{\x \in \mathbb{A}^n(\mathfrak{o}_{\mathfrak{p}}): \x \bmod \mathfrak{p} \not \in \mathscr{Z}(\FF_\mathfrak{p}) \}
\subset \pi(X(k_\mathfrak{p}))$$
for all sufficiently large primes ideals $\mathfrak{p}$.
Condition \eqref{eqn:lim_M_affine} therefore follows from Lemma~\ref{lem:codim_2_affine}. 
Finally, the measurability conditions at the finite places are a straightforward consequence of 
Lemma \ref{lem:Tarski}. This concludes the proof of  Theorem \ref{thm:Ekedahl_affine}. \qed

%\section{Elements of $\Br(X_\eta)$} \label{s:Br}
%T [see brauer-element.tex]

\section{Proof of Theorem \ref{thm:H1} --- strategy}\label{s:cow-herder}

In this section we describe the structure of the proof of Theorem~\ref{thm:H1} and introduce some notation.  Let us begin by giving an overview of the main ideas of the proof. We focus our discussion on the case where there is a non-trivial element of the group $\H^1(K, \Pic X_{\bar{\eta}})$, as this case is more difficult. 
Along the way, we shall indicate the changes that need to be made when working with 
a non-trivial element of $\Br X_\eta / \Br K$.
Our strategy is to show that, for most $P \in \PP^n(k)$, the element we work with  specialises to give a
class in $\Br X_P / \Br k$, which can obstruct weak approximation on $X_P$.  We need to understand how this obstruction varies with $P$.

Suppose that we have a variety $V$ over a number field $k$ and an element $\cA \in\Br V$.  
According to work of Bright \cite[\S 5.2]{bright'},  if there is a prime $\fp$ of $k$ with $\Norm \fp$ sufficiently large and such that $\cA$ is ramified at $\fp$, then 
the evaluation map $\mathrm{ev}_\cA \colon V(k_\fp) \to \Br k_\fp$ takes many values.  In particular, it follows that $\cA$ gives an obstruction to weak approximation on $V$.

Our plan will be to understand how this condition varies in the family of varieties considered in Theorem~\ref{thm:H1}.
In the family $\pi \colon X \to \PP^n$, there is a dense open subset $U \subset \PP^n$ above which the fibres are smooth and  proper varieties over  $k$.  These are the ``good'' members of the family.  The complement of $U$ is a union
$\PP^n \setminus U = \bigcup_i S_i$ of 
 locally closed components.
Above the generic point of each component $S_i$, the fibre of $X$ is ``bad'' in some way.  We shall think of the components $S_i$ as the different possible types of bad reduction of a smooth variety in our family, as follows. Let $\fo$ be the ring of integers of $k$.  If $P \in U(k)$ is a point, then $P$ extends uniquely to a point of $\PP^n(\fo)$ and so it makes sense to talk about the reduction of $P$ modulo any prime $\fp$ of $\fo$.  The reduction of $P$ will usually land in $U$, but at finitely many bad primes it will  land in one of the components $S_i$.  At such a prime, the smooth $k$-variety $X_P$ will have bad reduction, of a type corresponding to the component $S_i$.  The crucial observation, recorded in Proposition~\ref{cor:sieve}, is that if some $S_i$ has codimension $1$ in $\PP^n$, then almost every variety $X_P$ (for $P \in U(k)$) has \emph{some} prime $\fp$ at which it has the bad reduction type corresponding to $S_i$. This bad reduction will then be used to force a Brauer--Manin obstruction to weak approximation for $X_P$ at $\fp$.

To illustrate matters, consider the family of diagonal cubic surfaces~\eqref{eq:onion-collector}.  The open subset $U \subset \PP^n$ over which the fibres are smooth is given by $a_0a_1a_2a_3 \neq 0$.  Above each of the four divisors $\{ a_i =0 \}$, the generic fibre is a cone; above the closed subset where more than one of the $a_i$ vanishes, the reduction is something more singular.  Here it is easy to show by elementary methods that 100\% of the smooth members of the family have the property that there exists a prime $\p$ dividing just one of $a_0,a_1,a_2,a_3$, so have bad reduction at $\p$ which is a cone. This type of bad reduction is enough to guarantee that there is a Brauer--Manin obstruction to weak approximation (as first noticed in \cite[\S 5]{ct-k-s} in the context of the Brauer--Manin obstruction to the Hasse principle). 

Returning to the proof of  Theorem~\ref{thm:H1}, the strategy of the proof can be summarised in the following steps.
\begin{enumerate}
\item\label{step1} For any prime divisor $D \subset \PP^n$, and any class $\alpha \in \H^1(K, \Pic X_{\bar{\eta}})$, we define the \emph{residue} of $\alpha$ at $D$; when that residue is non-zero, we will say that $\alpha$ is \emph{ramified} at $D$.
\item\label{step2} We show that, under the hypotheses of Theorem~\ref{thm:H1}, every non-zero element of $\H^1(K, \Pic X_{\bar{\eta}})$ is ramified at some prime divisor $D \subset \PP^n$, necessarily contained in the complement of $U$.
\item\label{step3} Fix an element $\alpha \in \H^1(K, \Pic X_{\bar{\eta}})$ which is ramified at a prime divisor $D$.
Suppose that $P \in U(k)$ is such that the smooth variety $X_P$ has bad reduction at a prime $\fp$, of the type corresponding to the divisor $D$.  (In other words, the reduction of $P$ modulo $\p$ lies in $D$, in a suitably ``generic'' way.)  Then, using the result of~\cite[\S 5.2]{bright'} described above, we prove that there is an obstruction to weak approximation on $X_P$. This obstruction is given by an element of $\Br X_P$ obtained by ``specialising'' $\alpha$ in an appropriate way.
\item\label{step4} Finally we show that the previous statement applies to 100\% of $P \in U(k)$.
\end{enumerate}
Steps~\ref{step2}, \ref{step3} and~\ref{step4} are independent.  Step~\ref{step1} will be addressed in \S\S\ref{sec:res}--\ref{sec:resH1}, Step~\ref{step2} in \S\ref{s:bee-keeper}, Step~\ref{step3} in \S\ref{s:bull-rider} and Step~\ref{step4} in \S\ref{s:sieve}.

We proceed to  state formally the primary waypoints in  each step of the proof.  Steps~\ref{step1}--\ref{step3} are purely algebraic and apply in a wider context than that of Theorem~\ref{thm:H1}. It will be convenient to specify here the particular families of schemes that we will be working with for most of the remainder of this paper. 

\begin{condition}\label{cond}
$\pi \colon X \to B$ is a flat, surjective morphism of finite type between regular integral Noetherian schemes. If $K$ is the function field  $K$ of $B$ and $\eta \colon \Spec K \to B$ is the generic point, then the generic fibre $X_\eta$ is smooth, proper and geometrically connected with torsion-free geometric Picard group.  The fibre of $\pi$ at each codimension-1 point of $B$ is geometrically integral.  The fibre of $\pi$ at each codimension-2 point of $B$ has a geometrically reduced component.
\end{condition}

The assumption that the generic fibre have torsion-free geometric Picard
group is very natural when dealing with weak approximation. Indeed, if $V$ is an everywhere locally soluble smooth proper variety over a number field $k$ with $\Pic \bar{V}$ not torsion-free, then $V$ is not geometrically simply connected and hence a result of Minchev (see \cite[Thm.~2.4.5]{harariWA}) implies that $V$ always fails weak approximation.
In particular, the non-trivial case of Theorem \ref{thm:H1} is  when the geometric Picard group is torsion-free.
This condition also simplifies many of the constructions performed in  \S \ref{sec:Pic}.

Throughout we will fix an algebraic closure $\bar{K}$ of $K$ and also a geometric point $\bar{\eta} \colon \Spec \bar{K} \to B$ lying over $\eta$.  Our chief interest lies in the case that $X$ and $B$ are varieties over a number field $k$. 
All cohomology will be \'etale cohomology.

\subsubsection*{Step \ref{step1}}
The first step of the proof of Theorem~\ref{thm:H1} is to define the residue of an element of the group $\H^1(K, \Pic X_{\bar{\eta}})$ at a prime divisor $D \subset B$.   Recall that there is an exact sequence
\[
\Br K \to \Br_1 X_{\eta} \to \H^1(K, \Pic X_{\bar{\eta}}) \to \H^3(K, \Gm)
\]
coming from the Hochschild--Serre spectral sequence.  For any prime divisor $Z \subset X$, Grothendieck defined a residue map
\[
\partial_Z \colon \Br X_\eta \to \H^1(\ff{Z}, \QZ).
\]
Grothendieck's purity theorem for the Brauer group \cite[III, Thm.~6.1]{Grothendieck_Brauer} shows that $\Br X \subset \Br X_\eta$ is precisely the subgroup consisting of those classes having zero residue at every vertical prime divisor $Z$.  As we've already seen, however,  work of Uematsu~\cite{uematsu} shows that $\Br X_\eta$ is not as useful as it might appear,  since interesting elements of $\H^1(K, \Pic X_{\bar{\eta}})$ do not necessarily lift to $\Br_1 X_\eta$.
Our aim is therefore to define a residue map directly on $\H^1(K, \Pic X_{\bar{\eta}})$, compatible with that on $\Br X_\eta$, and prove a corresponding purity theorem.

It turns out to be easier to work with the group $\H^0(K, \H^2(X_{\bar{\eta}}, \mmu_n))$, where $n>1$ is a fixed integer which is invertible on $B$; this allows us to treat elements of $\H^1(K,\Pic X_{\bar{\eta}})$ and elements of $\Br X_\eta$ in a uniform manner.
Let us show how this group appears.
The exact sequence
\[
0 \to \Pic X_{\bar{\eta}} \xrightarrow{n} \Pic X_{\bar{\eta}} \to (\Pic X_{\bar{\eta}})/n \to 0
\]
gives rise to an exact sequence in cohomology
\[
\H^0(K, \Pic X_{\bar{\eta}}) \to \H^0(K, (\Pic X_{\bar{\eta}})/n) \to \H^1(K, \Pic X_{\bar{\eta}})[n] \to 0.
\]
Thus every $n$-torsion class in $\H^1(K, \Pic X_{\bar{\eta}})$ is represented by a Galois-fixed element of $(\Pic X_{\bar{\eta}})/n$.  On the other hand, the Kummer sequence gives an exact sequence
\[
0 \to (\Pic X_{\bar{\eta}})/n \to \H^2(X_{\bar{\eta}}, \mmu_n) \to \Br X_{\bar{\eta}}[n] \to 0
\]
and so a Galois-fixed element of $(\Pic X_{\bar{\eta}})/n$ gives rise to a Galois-fixed element of $\H^2(X_{\bar{\eta}}, \mmu_n)$.

To any codimension-$1$ point $d$ of $B$, we will associate a \emph{relative residue map}
\[
\rho_d \colon \H^0(K, \H^2(X_{\bar{\eta}}, \mmu_n)) \to \H^0(\ff{d}, \H^1(X^\sm_{\bar{d}}, \Zn))
\]
in Proposition \ref{prop:relpur}, which is
closely related to the usual residue map for the Brauer group.  Here $X^\sm_{\bar{d}}$ denotes the non-singular locus of the variety $X_{\bar{d}}$, where $\bar{d}$ is a geometric point lying over $d$.  In Proposition~\ref{prop:R2seq} we will prove a purity result:  the sequence
\[
0 \to \H^0(B, \R^2 \pi_* \mmu_n) \to \H^0(K, \H^2(X_{\bar{\eta}}, \mmu_n)) \xrightarrow{\oplus \rho_{d}} \bigoplus_{d \in B^{(1)}} \H^0(\ff{d}, \H^1(X_{\bar{d}}^\sm, \Zn))
\]
is exact.  This completes Step~\ref{step1} of the proof.

\subsubsection*{Step~\ref{step2}} The second stage of the proof of Theorem~\ref{thm:H1} is to show that, given any non-trivial 
element of $\H^1(K, \Pic X_{\bar{\eta}})$ giving rise to $\alpha \in \H^0(K,
\H^2(X_{\bar{\eta}},\mmu_n))$, there exists a codimension-1 point $d \in B^{(1)}$ satisfying $\rho_d(\alpha) \neq 0$.  This will be achieved in \S \ref{s:bee-keeper}, where it will be established under the assumptions:
\begin{itemize}
\item
$B=\PP^n$, 
\item
$\Br \bar X=0$,
\item $\H^1(k, \Pic \bar{X})=0$,
\item $\H^2(k,\Pic \bar{B}) \to 
\H^2(k,\Pic \bar X) $ is injective. 
\end{itemize}
As input to Step~\ref{step3}, we need not only $\rho_d(\alpha) \neq 0$ but the stronger condition $\rho_d(\alpha)=n$.  However, this is easily arranged by starting with an element of prime order $n$ in $\H^1(K,\Pic X_{\bar{\eta}})$.%, so that $n$ is prime; then, given that $\rho_d(\alpha)$ has order dividing $n$, it has order $n$ as soon as it is non-zero.

If instead we have a non-trivial element of $\Br X_\eta / \Br K$, then 
in \S \ref{sec:Br} we 
  prove that a class in $\H^0(K,\H^2(X_{\bar{\eta}},\mmu_n))$ associated to it has non-trivial residue at some $d \in B^{(1)}$.

\subsubsection*{Step~\ref{step3}} 
Assume that $\pi:X\to B$  is a morphism of smooth varieties over a number field $k$ satisfying Condition \ref{cond}.
Let $U\subset B$ be a non-empty open subset over which $\pi$ is smooth and proper. 
For any point $P \in U(k)$ such that $X_P$ is everywhere locally soluble, and any element $\alpha \in \H^0(K, \H^2(X_{\bar{\eta}}, \mmu_n))$, we will define the specialisation $\spec_P(\alpha) \in \Br X_P[n]$, defined only up to elements of $\Br k[n]$.

To state the main ingredient in Step~\ref{step3} of the proof of Theorem~\ref{thm:H1}, we need to define some more terms.
By a \emph{model} of $\pi$ over $\fo$ we mean an integral scheme $\sB$, separated and of finite type over $\fo$, satisfying $\sB \times_{\fo} k = B$, together with an $\fo$-morphism $\sX \to \sB$, separated and of finite type, extending $\pi$.
Next, let $S$ be a regular integral Noetherian scheme of dimension $d$.
Two locally closed subschemes $Y,Z$ of codimensions $c$ and $d-c$ are said to \emph{meet transversely} at a closed point $P \in Y \cap Z$ if, in the local ring $\OO_{S,P}$, the ideals defining $Y$ and $Z$ together generate the maximal ideal of $\OO_{S,P}$.
Bearing all this in mind, the following result will be established in  \S \ref{s:bull-rider}.

\begin{proposition}\label{prop:H2}
Let $\alpha\in  \H^0(K, \H^2(X_{\bar{\eta}},\mmu_n))$
be a class with $n >1$ and suppose that there is a point   $d \in B^{(1)}$ such that the relative residue 
\[\rho_d(\alpha) \in \H^0(\ff{d},\H^1(X^\sm_{\bar{d}},\Zn))\] has order $n$.
Let $\sX \to \sB$ be a model of $\pi$ over $\fo$.
Denote by $\sD$ the Zariski closure of $d$ in $\sB$.  Then there is a dense open subset $\sV \subset \sD$ with the following property:

Let $P$ be a point of $U(k)$ such that $X_P$ is everywhere locally soluble, and suppose that the Zariski closure of $P$ in $\sB$ meets $\sV$ transversely at a closed point $s$.  Let $\p$ be the prime of $\fo$ over which $s$ lies.  Then
the evaluation map $X_P(k_\p) \to \Br k_\p[n]$ given by the algebra $\spec_P(\alpha)$ is surjective.
In particular, weak approximation fails for  $X_P$.
%the specialisation $\spec_P(\alpha) \in \Br X_P / \Br k$ is prolific at the place of $\fo$ over which $s$ lies.
\end{proposition}

This result says that if $X_P$ has a certain kind of bad reduction at $\fp$, then a Brauer--Manin obstruction to weak approximation occurs for $X_P$ at $\fp$.
Note that the conclusion of the proposition does not depend on the choice of $\spec_P(\alpha)$ modulo $\Br k[n]$.

\subsubsection*{Step~\ref{step4}}
The final step in the proof of Theorem \ref{thm:H1} is to show that, with polynomial decay,  100\% of the points $P$ in $U(k) \cap \pi(X(\A_k))$ satisfy the condition appearing in Proposition~\ref{prop:H2}.
Let $\pi:X \to \PP^n$ satisfy the conditions of Theorem~\ref{thm:H1} and let $\sX\to \sB = \PP^n_\fo$ be a model of $\pi$ over $\fo$. 
Let $\sD$ and  $\sV$ be as in the statement of Proposition \ref{prop:H2}.  The following 
result will be established in \S \ref{s:sieve} using the large sieve inequality.
%statement is a geometric application of Proposition~\ref{prop:fly-catcher}.

\begin{proposition}\label{cor:sieve}
Let $T=U(k)\cap \pi(X(\A_k))$ and let $T_{\text{trans}}$
be  the set of points $P\in T$ such that 
the  Zariski closure $\bar{P}$ in $\PP^n_\fo$ 
meets $\sV$ transversely in at least one point.
Then 
$$
\#\{P\in T: H(P)\leq B\}-\# \{ P \in T_{\text{trans}} : H(P) \le B\}
=O_{\sV} \left(\frac{B^{n+1}}{\log B}\right). 
$$
%\[
%\lim_{B\to \infty}\frac{\# \{ P \in T_{\text{trans}} : H(P) \le B\}}{\#\{P\in T: H(P)\leq B\}}
%=1.
%\]
%In particular the limit exists.
\end{proposition}

This result implies that the type of bad reduction encountered in Proposition~\ref{prop:H2} occurs $100\%$ of the time. This therefore  suffices to complete the proof of Theorem~\ref{thm:H1}.

\section{Elements of \texorpdfstring{$\H^1(K, \Pic X_{\bar{\eta}})$}{} and residues} \label{sec:Pic}

This section contains the algebraic part of the proof of Theorem~\ref{thm:H1}, an overview of which is given in~\S\ref{s:cow-herder}.

Let $\pi \colon X \to B$ be a map of schemes satisfying Condition~\ref{cond}.  Our first task is to define, for any positive integer $n$ invertible on $B$, a relative residue map
\[
\rho_d \colon \H^0(K, \H^2(X_{\bar{\eta}}, \mmu_n)) \to \H^0(\ff{d}, \H^1(X^\sm_{\bar{d}}, \Zn)).
\]
The group $\H^2(X_{\bar{\eta}}, \mmu_n)$ is the stalk at the generic point of the sheaf $\R^2 \pi_* \mmu_n$, so an element of $\H^0(K, \H^2(X_{\bar{\eta}}, \mmu_n))$ can be thought of as a generic section of $\R^2 \pi_* \mmu_n$.  In \S\ref{sec:res} we will investigate the sheaf $\R^2 \pi_* \mmu_n$ for an arbitrary morphism of regular schemes, and in particular look at the consequences for this sheaf of the absolute cohomological purity theorem.  In \S\ref{sec:generic} we return to the situation of Condition~\ref{cond}, define a residue map for each codimension-$1$ point of $B$, and prove the appropriate purity theorem.  In \S\ref{sec:resH1} we relate this back to the group $\H^1(K, \Pic X_{\bar{\eta}})$.

Having defined the residue maps, we turn in \S\ref{s:bee-keeper} to finding conditions on $X \to B$ which ensure that every non-trivial element of $\H^1(K, \Pic X_{\bar{\eta}})$ is ramified somewhere.  Finally, \S\ref{s:bull-rider} is devoted to the proof of Proposition~\ref{prop:H2}.

\subsection{Purity and relative residue maps}\label{sec:res}

In this section we prove some purity-type results on the sheaf $\R^2 \pi_* \mmu_n$ associated to a morphism $\pi \colon X \to B$ of regular schemes.  We will use the language of derived categories, which is particularly useful in the proof of Proposition~\ref{prop:relpur}(\ref{relpur1}).  All cohomology will be \'etale cohomology.

Let us first recall some tools relating the cohomology of a scheme to that of an open subscheme.  Let $X$ be a scheme, $i \colon Z \to X$ a closed immersion and $j \colon U \to X$ the inclusion of the complement of $Z$.  Let $i^!$ denote the functor taking a sheaf on $X_\et$ to its subsheaf of sections with support in $Z$, considered as a sheaf on $Z_\et$.  Let $n$ be a positive integer invertible on $X$, and let $\Lambda$ be the sheaf $\Zn$.  All the following general results will be stated for $\Lambda$ but will be equally applicable to any twist $\Lambda(r)$, because all the functors involved commute with twists.

As described in~\cite[Tag 0A45]{stacks-project}, there is a distinguished triangle
\begin{equation}\label{eq:triangle}
i_* \R i^! \Lambda \to \Lambda \to \R j_* \Lambda \to i_* \R i^! \Lambda[1]
\end{equation}
in $D^+(X_\et)$.
We will need to understand how the triangle~\eqref{eq:triangle} behaves functorially.  Let $f \colon X' \to X$ be a morphism, and denote by $U',Z',i',j'$ the base changes of $U,Z,i,j$ by $f$.  Applying $\R f_*$ to the triangle~\eqref{eq:triangle} for $X'$ gives the second row of the diagram
\begin{equation}\label{eq:reltriangle}
\begin{CD}
i_* \R i^! \Lambda @>>> \Lambda @>>> \R j_* \Lambda @>>> i_* \R i^! \Lambda[1] \\
@V{\alpha}VV @V{\beta}VV @V{\gamma}VV @V{\alpha[1]}VV \\
\R f_* i'_* \R (i')^! \Lambda @>>> \R f_* \Lambda @>>> \R f_* \R j'_* \Lambda @>>> \R f_* i'_* \R (i')^! \Lambda[1]
\end{CD}
\end{equation}
in $D^+(X_\et)$ in which both rows are distinguished triangles; let us define the vertical maps and show that the diagram commutes.
The map $\beta$ is the natural map $\Lambda \to \R f_* f^* \Lambda$ arising by adjunction (since the group scheme $\Zn$ is \'etale over $X$, we have $f^* \Lambda = \Lambda$).  The map $\gamma$ is the composition of the natural map 
$$\R j_* \Lambda \to \R j_* \R (f_U)_* f_U^*\Lambda
$$ 
given by adjunction with the natural isomorphism $\R j_* \R (f_U)_* \Lambda = \R f_* \R j'_* \Lambda$.  The middle square commutes because the composite morphism $\Lambda \to \R f_* \R j'_* \Lambda$ is given by the unit of the adjunction between $\phi_*$ and $\phi^*$, where $\phi = fj'=jf_U$; the two ways round the square correspond to viewing this adjunction as the composite of two adjunctions in two different ways.  By~\cite[Proposition~1.1.9]{BBD}, a map $\alpha$ exists making the diagram into a morphism of triangles; moreover, since we have $\Hom(i_* \mathcal{F}, j_* \mathcal{G})=0$ for any $\mathcal{F} \in D^+(Z_\et)$ and $\mathcal{G} \in D^+(U_\et)$, the criterion given there shows that $\alpha$ is uniquely determined.

The absolute cohomological purity theorem, proved by Gabber~\cite{riou}, states the following.
\begin{theorem}[Gabber]\label{thm:purity}
Let $X$ be a regular scheme and $Z$ a regular closed subscheme of codimension $c$. Let $n$ be a positive integer invertible on $X$.  Then there is a canonical isomorphism $\Cl_i \colon \Lambda \to i^! \Lambda(c)[2c]$ in $D^+(Z_\et)$.
\end{theorem}

Shifting and twisting appropriately gives an isomorphism $\Lambda(-c)[-2c] \to i^! \Lambda$.
Combining this with the triangle~\eqref{eq:triangle} gives a distinguished triangle
\begin{equation}\label{eq:gystriangle}
i_* \Lambda(-c)[-2c] \to \Lambda \to \R j_* (\Lambda) \to i_* \Lambda(-c)[-2c+1].
\end{equation}
It follows from Proposition~2.3.4 of~\cite{riou} that, under the canonical identification
\[
\Hom_{D^+(X)}(i_* \Lambda(-c)[-2c], \Lambda) = \H^{2c}_Z(X,\Lambda(c)),
\]
the homomorphism $i_* \Lambda(-c)[-2c] \to \Lambda$ in~\eqref{eq:gystriangle} corresponds to the cycle class $\cl(Z) \in \H^{2c}_Z(X,\Lambda(c))$ as defined in~\cite{Cycle}.

\begin{corollary}\label{cor:purityseq}
Under the conditions of Theorem~\ref{thm:purity}, there are isomorphisms 
$$\H^p(X, \Lambda) \to \H^p(U, \Lambda)
$$ for $0 \le p \le 2c-1$, and a long exact sequence
\[
0 \to \H^p(X, \Lambda) \to \H^p(U, \Lambda) \to \H^{p-2c}(Z, \Lambda(-c)) \to \H^{p+1}(X, \Lambda) \to \dotsb.
\]
\end{corollary}
\begin{proof}
Apply the global sections functor $\R\Gamma_X$ to \eqref{eq:gystriangle} and take homology.
\end{proof}

The purity theorem has the following interpretation in terms of the higher direct images $\R^p j_* \Lambda$.
\begin{corollary}
Under the conditions of Theorem~\ref{thm:purity}, we have
$j_* \Lambda = \Lambda$; $\R^p j_* \Lambda=0$ for $0<p<2c-1$, and $\R^{2c-1} j_* \Lambda \cong i_* \Lambda(-c)$.
\end{corollary}
\begin{proof}
Take the long exact sequence in homology of the triangle~\eqref{eq:gystriangle}.
\end{proof}

An easy and well-known consequence is the following \emph{semi-purity} statement.
\begin{lemma}\label{lem:semipur}
Let $X$ be a regular, Noetherian, excellent scheme and $i \colon Z \to X$ any closed immersion of codimension $\ge c$.  Suppose that $n$ is invertible on $X$.  Then we have $\R^p i^! \Lambda = 0$ for $0 \le p < 2c$.
\end{lemma}
\begin{proof}
Let $j \colon U \to X$ be the inclusion of the complement of $Z$.  The triangle~\eqref{eq:triangle} shows that the conclusion is equivalent to the statement $j_* \Lambda = \Lambda$ and  $\R^p j_* \Lambda = 0$ for $0 < p < 2c-1$.
Since this statement does not depend on the scheme structure on $Z$, we may assume that $Z$ is reduced.  The purity theorem shows that the conclusion holds if $Z$ is regular, and so in particular if $\dim Z = 0$.  In the general case, we use Noetherian induction.  Let $S$ be the union of all irreducible components of $Z$ of codimension greater than $c$ in $X$, together with the non-regular locus of $Z$.  Then $S$ is a proper closed subset of $Z$.  Write $V = X \setminus S$ and consider the open immersions $j_1 \colon U \to V$ and $j_2 \colon V \to X$.  Then $Z \setminus S$ is a regular closed subscheme of codimension at least $c$ in $V$, and so by purity we have $\R^p (j_1)_* \Lambda=0$ for $0 < p < 2c-1$.  Since $S$ is of codimension $>c$ in $X$, we have $\R^q (j_2)_* \Lambda=0$ for $0 < q < 2c+1$ by induction.  The Leray spectral sequence for the composition $j = j_2 \circ j_1$ then gives the desired conclusion.
\end{proof}

Now let us understand how the triangle~\eqref{eq:gystriangle} behaves under base change.  Suppose that $f \colon X'  \to X$ is a morphism of regular schemes, that $Z \subset X$ is a regular closed subscheme of codimension $c$, and that the inverse image $Z' = f^{-1}(Z) \subset X'$ is also regular and of codimension $c$.  Combining Theorem~\ref{thm:purity} with the diagram~\eqref{eq:reltriangle} we obtain a morphism of triangles as follows, with notation as above.
\[
\begin{CD}
i_* \Lambda(-c)[-2c] @>>> \Lambda @>>> \R j_* \Lambda @>>> i_* \Lambda(-c)[-2c+1] \\
@V{\alpha'}VV @V{\beta}VV @V{\gamma}VV @V{\alpha'[1]}VV \\
\R f_* i'_* \Lambda(-c)[-2c] @>>> \R f_* \Lambda @>>> \R f_* \R j'_* \Lambda @>>> \R f_* i'_* \Lambda(-c)[-2c+1]
\end{CD}
\]

As before, the morphism $\alpha'$ is uniquely determined.  There is an obvious candidate for the morphism $\alpha'$, namely that obtained by applying $i_*$ to the natural morphism $\phi \colon \Lambda(-c)[-2c] \to \R (f_Z)_* \Lambda(-c)[-2c]$.  Let us show that this is indeed the same morphism; to do so, we just need to show that the diagram commutes if we put $i_* \phi$ in place of $\alpha'$.  But then the two routes round the left-hand square are the homomorphisms defined by the classes $f^* \cl(Z)$ and $\cl(Z')$ respectively in $\H^{2c}_{Z'}(X',\Lambda(c))$, which agree.

We now return to the situation of interest: $\pi \colon X \to B$ is a flat morphism, and we are interested in $\R^2 \pi_* \mmu_n$.  The first result is concerned with removing a closed subscheme of codimension at least $2$ in $X$.

\begin{lemma}\label{lem:R2pur}
Let $\pi \colon X \to B$ be a morphism of regular, excellent, Noetherian schemes, and let $Z \subset X$ be a closed subset of codimension at least $2$.  Denote by $\pi'$ the restriction of $\pi$ to $X \setminus Z$.  Then, for any $p \le 2$, the natural map $\R^p \pi_* \mmu_n \to \R^p \pi'_* \mmu_n$ is an isomorphism.
\end{lemma}
\begin{proof}
Write $j$ for the inclusion of $X \setminus Z$ into $X$.  Lemma~\ref{lem:semipur} shows $j_* \mmu_n = \mmu_n$ and $\R^p j_* \mmu_n=0$ for $p=1,2$.  The Leray spectral sequence for the composition $\pi' = \pi j$ then gives the result.
\end{proof}

Next, we consider removing a closed subscheme of $B$.

\begin{proposition}\label{prop:relpur}
Let $\pi \colon X \to B$ be a flat morphism of regular, excellent, Noetherian schemes.  Let $n$ be an integer invertible on $B$.  
Let $S \subset B$ be a reduced closed subscheme, and let $Z$ denote the inverse image $\pi^{-1}(S)$.  Let $U$ be the complement of $Z$ in $X$.  Suppose that the restriction $\pi_U \colon U \to (B \setminus S)$ satisfies $(\pi_U)_* \mmu_n = \mmu_n$ and $\R^1 (\pi_U)_* \mmu_n = 0$.  Let $\pi_Z \colon Z \to S$ be the restriction of $\pi$ to $Z$.
\begin{enumerate}
\item\label{relpur3}
 Suppose that $S$ has codimension at least $3$ in $B$.  Then the natural map $\H^0(B, \R^2 \pi_* \mmu_n) \to \H^0(B \setminus S, \R^2 \pi_* \mmu_n)$ is an isomorphism.
\item\label{relpur2}
 Suppose that $S$ is regular of codimension $2$ in $B$, that $Z$ is also regular of codimension $2$, and that the natural map $(\Zn)_S \to (\pi_Z)_* (\Zn)_Z$ is injective.  Then the natural map $\H^0(B, \R^2 \pi_* \mmu_n) \to \H^0(B \setminus S, \R^2 \pi_* \mmu_n)$ is an isomorphism.
\item\label{relpur1}
 Suppose that $S$ is regular of codimension $1$ in $B$, that $Z$ is also regular of codimension $1$, and that the natural map $(\Zn)_S \to (\pi_Z)_* (\Zn)_Z$ is an isomorphism.
Then there is an exact sequence
\begin{equation}\label{eq:relres}
0 \to \H^0(B, \R^2 \pi_* \mmu_n) \to \H^0(B \setminus S, \R^2 \pi_* \mmu_n) \xrightarrow{\rho_S} \H^0(S, \R^1 (\pi_Z)_* \Zn)
\end{equation}
with the following properties:
\begin{itemize}
\item the following diagram commutes:
\begin{equation}\label{eq:comm1}
\begin{CD}
\H^2(X,\mmu_n) @>>> \H^2(X \setminus Z,\mmu_n) @>{\partial_Z}>> \H^1(Z, \Zn) \\
@VVV @VVV @VVV \\
\H^0(B, \R^2 \pi_* \mmu_n) @>>> \H^0(B \setminus S, \R^2 \pi_* \mmu_n) @>{\rho_S}>> \H^0(S, \R^1 (\pi_Z)_* \Zn)
\end{CD}
\end{equation}
where the top row is part of the long exact sequence of Corollary~\ref{cor:purityseq} and
the vertical maps are the edge maps from the Leray spectral sequence for $\pi$;
\item if $B' \to B$ is a morphism such that the hypotheses above still hold after base change to $B'$, then the exact sequence~\eqref{eq:relres} and the diagram~\eqref{eq:comm1} are functorial with respect to $B' \to B$.
%\item if $B' \to B$ is a morphism such that the hypotheses above hold after base change to $B'$, then the resulting diagram
%\begin{equation}\label{eq:comm2}
%\begin{CD}
%\H^0(B, \R^2 \pi_* \mmu_n) @>>> \H^0(B \setminus D, \R^2 \pi_* \mmu_n) @>>> \H^0(D, \R^1 \pi_* \Zn) \\
%@VVV @VVV @VVV \\
%\H^0(B', \R^2 \pi_* \mmu_n) @>>> \H^0(B' \setminus D', \R^2 \pi_* \mmu_n) @>>> \H^0(D', \R^1 \pi_* \Zn)
%\end{CD}
%\end{equation}
%commutes, where we set $D' = D \times_B B'$ and $\pi'$ is the base change of $\pi$ by $B' \to B$.
\end{itemize}
\end{enumerate}
In all three cases, we additionally have $\pi_* \mmu_n = \mmu_n$ and $\R^1 \pi_* \mmu_n = 0$.
\end{proposition}

The map $\rho_S$ of Proposition~\ref{prop:relpur}(\ref{relpur1}) is called the \emph{relative residue map} associated to the divisor $S \subset B$.

\begin{proof}
We give names to all the morphisms involved as follows
\[
\begin{CD}
Z @>>{i'}> X @<<{j'}< U \\
@V{\pi_Z}VV @V{\pi}VV @VV{\pi_U}V \\
S @>>{i}> B @<<{j}< B \setminus S \,.
\end{CD}
\]
The diagram~\eqref{eq:reltriangle} gives a commutative diagram
\[
\begin{CD}
i_* \R i^! \mmu_n @>>> \mmu_n @>>> \R j_* \mmu_n @>>> i_* \R i^! \mmu_n[1] \\
@VVV @VVV @VVV @VVV \\
i_* \R (\pi_Z)_* \R (i')^! \mmu_n @>>> \R \pi_* \mmu_n @>>> \R j_* \R (\pi_U)_* \mmu_n @>>> i_* \R (\pi_Z)_* \R (i')^! \mmu_n[1] \\
\end{CD}
\]
in $D^+(B_\et)$, where we have used the identities $\pi i' = i \pi_Z$ and $\pi j' = j \pi_U$.
By hypothesis, the vertical map $\mmu_n \to \R(\pi_U)_* \mmu_n$ on $B \setminus S$ induces isomorphisms on homology in degrees $<2$ and so fits into a distinguished triangle
\[
\mmu_n \to \R (\pi_U)_* \mmu_n \to \tau_{\ge 2} \R (\pi_U)_* \mmu_n \to \mmu_n[1]
\]
where $\tau_{\ge 2}$ is the canonical ``wise'' truncation functor.  We can extend the above diagram as follows.
\[
\begin{CD}
i_* \R i^! \mmu_n @>>> \mmu_n @>>> \R j_* \mmu_n @>>> i_* \R i^! \mmu_n[1] \\
@VVV @VVV @VVV @VVV \\
i_* \R (\pi_Z)_* \R (i')^! \mmu_n @>>> \R \pi_* \mmu_n @>>> \R j_* \R (\pi_U)_* \mmu_n @>>> i_* \R (\pi_Z)_* \R (i')^! \mmu_n[1] \\
@VVV @VVV @VVV @VVV \\
C @>>> A @>>> \R j_* \tau_{\ge 2} \R(\pi_U)_* \mmu_n @>>> C[1] \\
@VVV @VVV @VVV @VVV \\
i_* \R i^! \mmu_n[1] @>>> \mmu_n[1] @>>> \R j_* \mmu_n[1] @>>> i_* \R i^! \mmu_n[2]
\end{CD}
\]
Here all the rows and columns are distinguished triangles, and all squares commute apart from the bottom-right one, which anti-commutes.

In case~\ref{relpur3}, 
flatness of $\pi$ implies that $Z$ has codimension at least $3$ in $X$.
Lemma~\ref{lem:semipur} shows that $\R i^! \mmu_n$ and $\R (i')^! \mmu_n$ both have zero homology in degrees $<6$.  
Looking at the left-hand column shows that
$C$ has no homology in degrees $<5$.  Looking at the third row then shows that $A$ has no homology in degrees $<2$, so that $\mmu_n \to \R \pi_* \mmu_n$ induces isomorphisms on homology in degrees $<2$.  We deduce $\pi_* \mmu_n=\mmu_n$, that $\R^1 \pi_* \mmu_n$ vanishes and that $A$ is isomorphic to the truncation $\tau_{\ge 2} \R \pi_* \mmu_n$.  Applying the global sections functor $\R \Gamma_B$ to the third row of the diagram now gives an exact triangle
\[
\R \Gamma_B(C) \to \R \Gamma_B(\tau_{\ge 2} \R \pi_* \mmu_n) \to \R \Gamma_{B \setminus S} (\tau_{\ge 2} \R(\pi_U)_* \mmu_n) \to \R \Gamma_B(C)[1].
\]
Because $\Gamma_B = \H^0(B, -)$ is left exact, we have
\[
\H^2 (  \R \Gamma_B(\tau_{\ge 2} \R \pi_* \mmu_n) ) = \H^0(B,\H^2(\tau_{\ge 2} \R \pi_* \mmu_n)) = \H^0(B,( \R^2 \pi_* \mmu_n))
\]
and, compatibly, $\H^2(\R \Gamma_{B \setminus S}(\tau_{\ge 2} \R(\pi_U)_* \mmu_n)) = \H^0(B \setminus S, \R^2 (\pi_U)_* \mmu_n)$.  Also, $\R \Gamma_B(C)$ has no homology in degrees $<5$, so applying $\H^2$ to the above sequence shows that the natural map $\H^0(B, \R^2 \pi_* \mmu_n) \to \H^0(B \setminus S, \R^2 (\pi_U)_* \mmu_n)$ is an isomorphism.  Because $j$ is \'etale, the sheaf $\R^2(\pi_U)_* \mmu_n$ is simply the restriction of $\R^2 \pi_* \mmu_n$ to $B \setminus S$, so this completes the proof of case~\ref{relpur3}.

In case~\ref{relpur2}, Theorem~\ref{thm:purity} allows us to replace the left-hand column with
\[
i_* \Zn(-1)[-4] \to i_* \R(\pi_Z)_* \Zn(-1)[-4] \to C \to i_* \Zn(-1)[-3].
\]
The hypothesis shows that $C$ has no homology in degrees $<4$. Thus the  argument used  in case~\ref{relpur3} gives $\pi_* \mmu_n=\mmu_n$, together with the fact that $\R^1 \pi_* \mmu_n$ vanishes, and that $\H^0(B, \R^2 \pi_* \mmu_n) \to \H^0(B \setminus S, \R^2 \pi_* \mmu_n)$ is an isomorphism.

In case~\ref{relpur1}, Theorem~\ref{thm:purity} shows that the left-hand column is isomorphic to
\[
i_* \Zn[-2] \to i_* \R(\pi_Z)_* \Zn[-2] \to C \to i_* \Zn[-1].
\]
The hypothesis is equivalent to saying that $\Zn \to \R(\pi_Z)_* \Zn$ is a quasi-isomorphism in degree $0$, and so  $C$ is isomorphic to $i_* (\tau_{\ge 1} \R(\pi_Z)_* \Zn)[-2]$.  In particular, $C$ has no homology in degrees $<3$, so once again we have $\pi_* \mmu_n=\mmu_n$ and $\R^1 \pi_* \mmu_n=0$.
Now applying $\R\Gamma_B$ to the middle two rows of the diagram and taking homology gives a commutative diagram
\[
\begin{CD}
@. \H^2(X, \mmu_n) @>>> \H^2(U, \mmu_n) @>>> \H^1(Z, \Zn) \\
@. @VVV @VVV @VVV \\
0 @>>> \H^0(B, \R^2 \pi_* \mmu_n) @>>> \H^0(B \setminus S, \R^2 (\pi_U)_* \mmu_n) @>>> \H^0(S, \R^1 (\pi_Z)_* \Zn)
\end{CD}
\]
as sought.  The vertical maps are, by construction of the Leray spectral sequence, the claimed edge maps.
The functoriality statement follows from the functoriality of the triangle~\eqref{eq:gystriangle} as discussed above.
\end{proof}

\begin{remark}
For $(\Zn)_S \to (\pi_Z)_* (\Zn)_Z$ to be injective, it is sufficient that $\pi_Z$ be dominant (that is, the fibre of $\pi$ above each generic point of $S$ is non-empty).  For $(\Zn)_S \to (\pi_Z)_* (\Zn)_Z$ to be bijective, it is sufficient that the fibre of $Z$ above each generic point of $S$ be geometrically connected.
\end{remark}

\begin{remark}
Combining it with Lemma~\ref{lem:R2pur}, we see that~(\ref{relpur2}) of Proposition~\ref{prop:relpur} holds under the weaker condition that $Z$ contain an open subscheme satisfying the conditions of~(\ref{relpur2}).  For example, it is enough to assume that $Z$ is reduced and dominates $S$.
\end{remark}

\subsection{Application to generic elements}\label{sec:generic}

In order to use Proposition~\ref{prop:relpur} to study generic elements on families of schemes, we now suppose that $\pi: X\to B$ satisfies the properties recorded in Condition \ref{cond}.
Let $X^\sm$ denote the smooth locus of $\pi$, and $\pi^\sm$ the restriction of $\pi$ to $X^\sm$.  The morphism $\pi^\sm$ is flat, but not necessarily surjective.  Condition~\ref{cond} has the following consequences: the fibre of $\pi^\sm$ at every codimension-$1$ point of $B$ is of codimension $1$ in $X$ and geometrically integral; and the fibre of $\pi^\sm$ at every codimension-$2$ point of $B$ is of codimension $2$ in $X$ and non-empty.  Thus it will allow us to apply Proposition~\ref{prop:relpur} to $\pi^\sm$.  Also, $X^\sm$ contains all points of codimension $1$ in $X$: for each such point either is in the generic fibre or lies above a codimension-1 point of $B$.  Thus the complement of $X^\sm$ is of codimension at least $2$ in $X$.  We will repeatedly use this fact with Lemma~\ref{lem:R2pur} to deduce $\R^i \pi_* \mmu_n = \R^i \pi^\sm_* \mmu_n$ for $i \le 2$.

Because $\pi$ is quasi-compact and quasi-separated, we have $(\R^i \pi_* \mmu_n)_{\bar{\eta}} = \H^i(X_{\bar{\eta}}, \mmu_n)$ for all $i$.  The conditions on $X_\eta$ ensure $\H^0(X_{\bar{\eta}}, \mmu_n)=\mmu_n$ and $\H^1(X_{\bar{\eta}}, \mmu_n)=0$ whenever $n$ is invertible in $K$.

To begin with, we use Proposition~\ref{prop:relpur} inductively to show that Condition~\ref{cond} implies $\R^1 \pi_* \mmu_n=0$.

\begin{lemma}\label{lem:R1zero}
Let $\pi \colon X \to B$ satisfy Condition~\ref{cond}.  Then we have $\pi_* \mmu_n = \mmu_n$ and $\R^1 \pi_* \mmu_n = 0$ for every $n$ invertible on $B$.
\end{lemma}
\begin{proof}
Since  $\pi$ is flat with geometrically irreducible generic fibre, the natural map $\mmu_n \to \pi_* \mmu_n$
is locally an isomorphism, and therefore an isomorphism. This proves the first part of the lemma.
%The equality $\pi_* \mmu_n = \mmu_n$ follows from the fact that $\pi$ is flat with geometrically irreducible generic fibre.  Indeed, let $V \to B$ be \'etale; then the generic fibre of $X \times_B V \to V$ is irreducible; since $X \times_B V \to V$ is flat, it follows that $X \times_B V$ is irreducible, so we have $\H^0(X \times_B V, \mmu_n) = \mmu_n$.  Thus the natural map $\mmu_n \to \pi_* \mmu_n$ is locally an isomorphism, and therefore an isomorphism.

Now we prove the second statement.  As pointed out above, the non-smooth locus of $\pi$ has codimension at least $2$ in $X$; by Lemma~\ref{lem:R2pur} it suffices to prove the claim for the smooth locus of $\pi$.
There is a non-empty open subset $V \subset B$ such that $X_V \to V$ is smooth and proper with geometrically connected fibres.  Because we have $(\R^1 \pi_* \mmu_n)_{\bar{\eta}} =0$, proper-smooth base change shows that $\R^1 \pi_* \mmu_n = 0$ holds on $V$.  Let $S$ denote the complement of $V$.  
We now use induction on the codimension of $S$.  If $S$ has codimension $\ge 3$ in $B$, then the claim follows immediately from Proposition~\ref{prop:relpur}(\ref{relpur3}).  If $S$ has codimension $\le 2$, then let $S'$ be the singular locus of $S$.  Considering the regular scheme $S \setminus S'$ as a closed subscheme of $B \setminus S'$ and applying Proposition~\ref{prop:relpur} shows that $\R^1 \pi_* \mmu_n = 0$ holds on $B \setminus S'$, and by induction on the whole of $B$.
\end{proof}

Now we turn to $\R^2 \pi_* \mmu_n$.  As in the case of elements of the Brauer group, it is useful to consider the residue map of Proposition~\ref{prop:relpur}(\ref{relpur1}) when $D$ is any prime divisor (not necessarily regular), by applying the construction above to the local ring $\OO_{B,D}$. 

Let $\pi \colon X \to B$ satisfy Condition~\ref{cond}, and let $D \subset B$ be a prime divisor with generic point $d$.
Let $\sX \to \Spec \OO_{B,d}$ denote the base change of $X$ to $\Spec \OO_{B,d}$, and $\sX^\sm$ its smooth locus.  Suppose that $n$ is invertible on $\Spec \OO_{B,d}$.  Applying Proposition~\ref{prop:relpur}(\ref{relpur1}) to $\sX^\sm \to \Spec \OO_{B,d}$ (with $S=d$) gives a relative residue map $\rho_d$ fitting into an exact sequence
\[
0 \to \H^0(\Spec \OO_{B,d}, \R^2 \pi^\sm_* \mmu_n) \to \H^0(K, \H^2(X_{\bar{\eta}}, \mmu_n)) \xrightarrow{\rho_d} \H^0(\ff{d}, \H^1(X_{\bar{d}}^\sm, \Zn)).
\]
Since the complement of $\sX^\sm$ is of codimension at least $2$ in $\sX$, Lemma~\ref{lem:R2pur} gives
\begin{equation}\label{eq:relresdiv}
0 \to \H^0(\Spec \OO_{B,d}, \R^2 \pi_* \mmu_n) \to \H^0(K, \H^2(X_{\bar{\eta}}, \mmu_n)) \xrightarrow{\rho_d} \H^0(\ff{d}, \H^1(X_{\bar{d}}^\sm, \Zn)).
\end{equation}

If both $D$ and $Z=\pi^{-1}(D)$ happen to be regular, it might appear that there is potential for confusion between the homomorphism $\rho_d$ described here and the $\rho_D$ obtained by applying Proposition~\ref{prop:relpur}(\ref{relpur1}) directly.  However, the two are related by the commutative diagram
\begin{equation}\label{eq:dD}
\begin{CD}
\H^0(B \setminus D, \R^2 \pi_* \mmu_n) @>{\rho_D}>> \H^0(D, \R^1 (\pi_Z)_* \Zn) \\
@VVV @VVV \\
\H^0(K, \H^2(X_{\bar{\eta}}, \mmu_n)) @>{\rho_d}>> \H^0(\ff{d}, \H^1(X^\sm_{\bar{d}}, \Zn)).
\end{CD}
\end{equation}

\begin{proposition}\label{prop:R2seq}
Let $\pi \colon X \to B$ satisfy Condition~\ref{cond}.  
Then, for every integer $n$ invertible on $B$, there is an exact sequence
\[
0 \to \H^0(B, \R^2 \pi_* \mmu_n) \to \H^0(K, \H^2(X_{\bar{\eta}}, \mmu_n)) \xrightarrow{\oplus \rho_{d}} \bigoplus_{d \in B^{(1)}} \H^0(\ff{d}, \H^1(X_{\bar{d}}^\sm, \Zn))
\]
where the sum is over all points of codimension $1$ in $B$.
\end{proposition}
\begin{proof}
That the sequence is a complex follows immediately from the fact that~\eqref{eq:relresdiv} is a complex for each $d$.

Let $\alpha$ be a class in $\H^0(K, \H^2(X_{\bar{\eta}}), \mmu_n)$.  We must show that $\rho_d(\alpha)=0$ for all but finitely many divisors $D$, and that if $\rho_d(\alpha)=0$ for all $d$ then $\alpha$ is the image of a unique element of $\H^0(B, \R^2 \pi_* \mmu_n)$.  Using the fact that all the cohomology groups involved commute with inverse limits of schemes, we have
\[
\H^0(K, \H^2(X_{\bar{\eta}}, \mmu_n)) = \H^0(\Spec K, \eta^* \R^2 \pi_* \mmu_n) = \lim_{\longrightarrow} \H^0(V, \R^2 \pi_* \mmu_n)
\]
where the limit is over all non-empty open subsets $V \subset B$.  So there is some non-empty open $V \subset B$ such that $\alpha$ lies in $\H^0(V, \R^2 \pi_* \mmu_n)$.  For any point $d \in V$, the morphism $\eta \to V$ factors through $\Spec \OO_{B,d}$, showing that $\alpha$ extends to $\H^0(\Spec \OO_{B,d}, \R^2 \pi_* \mmu_n)$ and so, by~\eqref{eq:relresdiv}, $\rho_d(\alpha)=0$.  This is true for all but finitely many $d \in B^{(1)}$.

Next we prove exactness in the middle.  Suppose that we have $\rho_d(\alpha)=0$ for all $d \in B^{(1)}$.  For any such $d$, \eqref{eq:relresdiv} shows that $\alpha$ lifts to $\H^0(\Spec \OO_{B,d}, \R^2 \pi_* \mmu_n)$, and so there is an open neighbourhood $V_d $ of $d$ such that $\alpha$ lifts to $\H^0(V_d, \R^2 \pi_* \mmu_n)$.  Let $V'$ be the union of all the $V_d$; by compactness, $V'$ is actually the union of finitely many $V_d$, so the sheaf property for $\R^2 \pi_* \mmu_n$ shows that $\alpha$ lifts to a class in $\H^0(V', \R^2 \pi_* \mmu_n)$.  
To complete the proof, we will show that $\H^0(B, \R^2 \pi_* \mmu_n) \to \H^0(V', \R^2 \pi_* \mmu_n)$ is an isomorphism.
Since $V'$ contains every point of codimension $1$ in $B$, the complement $S = B \setminus V'$ has codimension at least $2$.  If $S$ has codimension greater than $2$, then Proposition~\ref{prop:relpur}(\ref{relpur3}) completes the proof.
Otherwise, the singular locus of $S$ is of codimension at least $3$ in $B$, and so by Proposition~\ref{prop:relpur}(\ref{relpur3}) we may remove it and assume that $S$ is regular and of pure codimension $2$.
By Proposition~\ref{prop:relpur}(\ref{relpur2}) applied to $\pi^\sm$, we deduce that $\H^0(B, \R^2 \pi^\sm_* \mmu_n) \to \H^0(V', \R^2 \pi^\sm_* \mmu_n)$ is an isomorphism.
But the complement of $X^\sm$ has codimension at least $2$, so Lemma~\ref{lem:R2pur} completes the proof.

Finally, we prove injectivity of $\H^0(B, \R^2 \pi_* \mmu_n) \to \H^0(K, \H^2(X_{\bar{\eta}}, \mmu_n))$.  Let $\alpha$ lie in the kernel of this map.  Let $d$ be point of codimension $1$ in $B$; by~\eqref{eq:relresdiv}, $\alpha$ restricts to $0$ in $\H^0(\Spec \OO_{B,d}, \R^2 \pi_* \mmu_n)$.  Therefore there is an open neighbourhood $V_d$ of $d$ such that $\alpha$ restricts to $0$ in $\H^0(V_D, \R^2 \pi_* \mmu_n)$.  As before, let $V'$ be the union of the $V_d$ for all $d \in B^{(1)}$; the sheaf property shows that $\alpha$ restricts to $0$ in $\H^0(V', \R^2 \pi_* \mmu_n)$.  But $\H^0(B, \R^2 \pi_* \mmu_n) \to \H^0(V', \R^2 \pi_* \mmu_n)$ is an isomorphism, so $\alpha$ is zero.
\end{proof}

\subsection{Application to $\H^1(K, \Pic X_{\bar{\eta}})$}\label{sec:resH1}

We now relate our results which we have proved on $\H^0(K, \H^2(X_{\bar{\eta}}, \mmu_n))$ back to the original group of interest, namely $\H^1(K, \Pic X_{\bar{\eta}})$.  

We begin by recalling some facts about the sheaf $\R^1 \pi_* \Gm$, all of which can be found in~\cite[\S 9.2]{kleiman}.  The sheaf $\R^1 \pi_* \Gm$ is the sheaf associated to the absolute Picard functor $V \mapsto \Pic(X \times_B V)$.  It is also the sheaf associated to the relative Picard functor $V \mapsto \Pic(X \times_B V) / \Pic V$.  The inverse image $\eta^* \R^1 \pi_* \Gm$ is simply the sheaf on $\Spec K$ associated to the Galois module $\Pic X_{\bar{\eta}}$.

\begin{lemma}\label{lem:picbr}
Let $\pi \colon X \to B$ satisfy Condition~\ref{cond}.  Then the natural map $$\R^1 \pi_* \Gm \to \eta_* \eta^* \R^1 \pi_* \Gm$$ is an isomorphism, and the natural map $\R^2 \pi_* \Gm \to \eta_* \eta^* \R^2 \pi_* \Gm$ is injective.
\end{lemma}
\begin{proof}
Let $V \to B$ be an \'etale morphism, and write $X_V$ for $X \times_B V$.  We claim that the sequence
\[
\Pic V \to \Pic X_V \to \Pic (X_V)_\eta \to 0
\]
is exact.  Indeed, the second arrow is surjective because $X_V$ is regular, so Weil and Cartier divisors coincide; and any Weil divisor on $(X_V)_\eta$ can be extended to $X_V$ by taking its Zariski closure.  Now let us show exactness at $\Pic X_V$.  If $D$ is a divisor on $X_V$ which becomes principal on restriction to $(X_V)_\eta$, say $D = (f)$ for some rational function on $(X_V)_\eta$, then we can consider $f$ as a rational function on $X_V$ and see that the divisor $D - (f)$ does not meet the generic fibre of $X_V \to V$, and therefore is vertical (that is, supported on fibres of $X_V \to V$).  But, by Condition~\ref{cond}, the fibres of $X_V$ above codimension-$1$ points of $V$ are integral, so every vertical divisor on $X_V$ is the pull-back of a divisor on $V$, showing the required exactness.
Thus we obtain an isomorphism between the presheaves
\[
V \mapsto \Pic X_V / \Pic V \quad \text{and} \quad V \mapsto \Pic (X_V)_\eta
\]
which, on sheafifying, gives an isomorphism $\R^1 \pi_* \Gm \to \eta_* \eta^* \R^1 \pi_* \Gm$ as claimed.

Since $X_V$ is regular, the restriction map $\Br X_V \to \Br (X_V)_\eta$ is injective (looking separately at each connected component, it suffices to prove this for $X_V$ integral, when it follows from~\cite[IV, Cor.~2.6]{milne}).  Sheafifying, we see that $\R^2 \pi_* \Gm \to \eta_* \eta^* \R^2 \pi_* \Gm$ is injective, as required.
\end{proof}

Recall from \S \ref{s:cow-herder} that to pass from elements of $\H^1(K, \Pic X_{\bar{\eta}})[n]$ to  $\H^0(K, \H^2(X_{\bar{\eta}}, \mmu_n))$, we used the exact sequence
\begin{equation}\label{eq:picnseq}
\H^0(K, \Pic X_{\bar{\eta}}) \to \H^0(K, \Pic X_{\bar{\eta}}/n) \to \H^1(K, \Pic X_{\bar{\eta}})[n] \to 0
\end{equation}
arising from the multiplication-by-$n$ homomorphism on $\Pic X_{\bar{\eta}}$, together with the injection $\Pic X_{\bar{\eta}}/n \to \H^2(X_{\bar{\eta}}, \mmu_n)$ coming from the Kummer sequence.  Whilst the resulting class depends on the initial choice of lift; the following lemma says that its residues do not.

\begin{lemma}\label{lem:kernel1}
Let $\pi \colon X \to B$ satisfy Condition~\ref{cond} and let $n$ be invertible on $B$.  Then the kernel of the composite homomorphism
\[
\H^0(K, \Pic X_{\bar{\eta}}/n) \to \H^0(K, \H^2(X_{\bar{\eta}}, \mmu_n)) \xrightarrow{\oplus \rho_{d}} \bigoplus_{d \in B^{(1)}} \H^0(\ff{d}, \H^1(X_{\bar{d}}^\sm, \Zn))
\]
is the image of the natural map $\H^0(B, (\R^1 \pi_* \Gm)/n) \to \H^0(K, \Pic X_{\bar{\eta}}/n)$.  It contains the image of $\H^0(K, \Pic X_{\bar{\eta}}) \to  \H^0(K, \Pic X_{\bar{\eta}}/n)$.
\end{lemma}
\begin{proof}
The Kummer sequence gives a short exact sequence
\[
0 \to (\R^1 \pi_* \Gm)/n \to \R^2 \pi_* \mmu_n \to (\R^2 \pi_* \Gm)[n] \to 0
\]
of sheaves on $B$.  Taking global sections on $B$ and on $\eta$ gives a commutative diagram with exact rows as follows.
\[
\begin{CD}
0 @>>> \H^0(B, (\R^1 \pi_* \Gm)/n) @>>> \H^0(B, \R^2 \pi_* \mmu_n) @>>> \H^0(B, (\R^2 \pi_* \Gm)[n]) \\
@. @V{\alpha}VV @V{\beta}VV @V{\gamma}VV \\
0 @>>> \H^0(K, \Pic X_{\bar{\eta}}/n) @>{i}>> \H^0(K, \H^2(X_{\bar{\eta}}, \mmu_n)) @>>> \H^0(K, \Br X_{\bar{\eta}} [n]).
\end{CD}
\]
By Proposition~\ref{prop:R2seq}, the kernel in question is $i^{-1}(\im \beta)$.  This clearly contains $\im \alpha$.  But $\gamma$ is injective by Lemma~\ref{lem:picbr}, and a diagram-chase then shows $i^{-1}(\im \beta)$ is contained in $\im \alpha$, proving the first claim.

For the second claim, we have a commutative diagram
\[
\begin{CD}
\H^0(B, \R^1 \pi_* \Gm) @>>> \H^0(B, (\R^1 \pi_* \Gm)/n) \\
@VVV @VVV \\
\H^0(K, \Pic X_{\bar{\eta}}) @>>> \H^0(K, \Pic X_{\bar{\eta}}/n).
\end{CD} 
\]
The left-hand vertical arrow is an isomorphism, by Lemma~\ref{lem:picbr}, and the claim follows easily.
\end{proof}

It follows from Lemma \ref{lem:kernel1} that there is a well-defined induced homomorphism
\begin{equation}\label{eq:H1picres}
\H^1(K, \Pic X_{\bar{\eta}})[n] \xrightarrow{\oplus \rho_{d}} \bigoplus_{d \in B^{(1)}} \H^0(\ff{d}, \H^1(X_{\bar{d}}^\sm, \Zn)).
\end{equation}
%T [something motivational about we want to prove that this is injective].

\begin{lemma}\label{lem:kernel}
The kernel of the homomorphism~\eqref{eq:H1picres} is the image of the natural map $\H^1(B, \R^1 \pi_* \Gm)[n] \to \H^1(K, \Pic X_{\bar{\eta}})[n]$.
\end{lemma}
\begin{proof}
By Lemma~\ref{lem:R1zero} we have $\R^1 \pi_* \mmu_n = 0$ and so the Kummer sequence shows that the multiplication-by-$n$ map on $\R^1 \pi_* \Gm$ is injective.  We therefore have a short exact sequence
\[
0 \to \R^1 \pi_* \Gm \xrightarrow{n} \R^1 \pi_* \Gm \to (\R^1 \pi_* \Gm)/n \to 0
\]
of sheaves on $B$, giving rise to an exact sequence in cohomology
\[
\H^0(B, \R^1 \pi_* \Gm) \to \H^0(B, (\R^1 \pi_* \Gm)/n) \to \H^1(B, \R^1 \pi_* \Gm)[n] \to 0
\]
compatible with~\eqref{eq:picnseq}.
Together, these give a commutative square
\[
\begin{CD}
\H^0(B, (\R^1 \pi_* \Gm)/n) @>>> \H^0(K, \Pic X_{\bar{\eta}}/n) \\
@VVV @VVV \\
\H^1(B, \R^1 \pi_* \Gm)[n] @>>> \H^1(K, \Pic X_{\bar{\eta}})[n]
\end{CD}
\]
in which both vertical maps are surjective.
By Lemma~\ref{lem:kernel1}, the kernel of the homomorphism~\eqref{eq:H1picres} is the image of the composite map in this square, which is equal to the image of the bottom arrow.
\end{proof}

\subsection{Existence of ramified elements}\label{s:bee-keeper}

We now return to the situation in which $X$ and $B$ are varieties over a number field $k$.  The purpose of this section is to give conditions under which every non-trivial element of $\H^1(K, \Pic X_{\bar \eta})$ is ramified along some divisor in $B$.  According to Lemma~\ref{lem:kernel}, this means understanding the group $\H^1(B, \R^1 \pi_* \Gm)$.  In what follows, we write $\bar{B}$ and $\bar{X}$ for the base changes of $B$ and $X$, respectively, to a fixed algebraic closure $\bar{k}$ of $k$.

\begin{lemma}\label{lem:H1R1}
Let $\pi \colon X \to B$ be a morphism of smooth proper varieties over a field $k$ satisfying Condition~\ref{cond}.  Suppose that we have $\Br X=\Br k$ and that the natural map $\H^3(B,\Gm) \to \H^3(X,\Gm)$ is injective.  Then $\H^1(B, \R^1 \pi_* \Gm)=0$.
\end{lemma}
\begin{proof}
Consider the Leray spectral sequence $E_2^{p,q}=\H^p(B, \R^q \pi_* \Gm) \Rightarrow \H^{p+q}(X, \Gm)$.  There is a homomorphism $d \colon \H^1(B, \R^1 \pi_* \Gm) \to \H^3(B, \Gm)$.  The hypothesis $\Br X = \Br k$ implies that the kernel of $d$ is trivial.  The image of $d$ lies in the kernel of $\H^3(B,\Gm) \to \H^3(X,\Gm)$, which is also trivial by hypothesis.  This gives the claimed vanishing $\H^1(B, \R^1 \pi_* \Gm)=0$.
\end{proof}

\begin{lemma}\label{lem:H3}
Let $\pi \colon X \to B$ be a morphism of smooth proper varieties over a number field $k$ satisfying Condition~\ref{cond}.  Suppose further that $B$ is such that the groups $\Br \bar{B}$ and $\H^3(\bar{B}, \Gm)$ both vanish.
\begin{enumerate}
\item If $\H^2(k, \Pic \bar{B}) \to \H^2(k, \Pic \bar{X})$ is not injective, then neither is $\H^3(B,\Gm) \to \H^3(X,\Gm)$.
\item If $\H^2(k, \Pic \bar{B}) \to \H^2(k, \Pic \bar{X})$ is injective and $\Br X \to \H^0(k, \Br \bar{X})$ is surjective, then $\H^3(B,\Gm) \to \H^3(X,\Gm)$ is injective.
\end{enumerate}
\end{lemma}
\begin{proof}
Because $k$ is a number field, we have $\H^3(k, \Gm)=0$.
Taking into account the assumptions on $B$,
the Hochschild--Serre spectral sequences for $\bar{X} \to X$ and $\bar{B} \to B$ give a commutative diagram as follows.
\[
\begin{CD}
@. 0 @>>> \H^2(k, \Pic \bar{B}) @>>> \H^3(B, \Gm) @>>> 0\\
@. @VVV @VVV @VVV \\
\Br X @>>> \H^0(k, \Br \bar{X}) @>>> \H^2(k, \Pic \bar{X}) @>>> \H^3(X, \Gm).
\end{CD}
\]
The result is now an easy application of the snake lemma.
\end{proof}

\begin{corollary}
Under the conditions of Theorem~\ref{thm:H1}, for every non-zero element $\alpha \in \H^1(K, \Pic X_{\bar{\eta}})$, there exists some $d \in B^{(1)}$ such that $\rho_d(\alpha)$ is non-zero.
\end{corollary}
\begin{proof}
The conditions include that $\H^2(k, \Pic \bar{B}) \to \H^2(k, \Pic \bar{X})$ is injective and that $\Br \bar{X}$ vanishes, so Lemma~\ref{lem:H3} shows that $\H^3(B,\Gm) \to \H^3(X,\Gm)$ is injective.  Then the conditions $\H^1(k, \Pic\bar{X})=0$ and $\Br \bar{X}=0$ together imply $\Br X = \Br k$, and Lemma~\ref{lem:H1R1} shows that $\H^1(B, \R^1 \pi_* \Gm)$ is zero.  By Lemma~\ref{lem:kernel}, this implies that the homomorphism~\eqref{eq:H1picres} is injective, giving the conclusion.
\end{proof}

Note that $\Br \bar{X}$ and $\H^1(k, \Pic \bar{X})$ both vanish when  $X$ is a $k$-rational variety.  On the other hand, the condition on $\ker \big(\H^2(k, \Pic \bar{B}) \to \H^2(k, \Pic \bar{X}) \big)$ appears less natural, and we briefly discuss it further.

As in the proof of Lemma~\ref{lem:picbr}, we have an exact sequence of $\Gal(\kbar/k)$-modules
\[
0 \to \Pic \bar{B} \to \Pic \bar{X} \to \Pic \bar{X}_\eta \to 0,
\]
where now $\Pic \bar{B} \to \Pic \bar{X}$ is injective because $X \to B$ is proper.  Here $\bar{X}_\eta$ is the generic fibre of the base change of $X$ to $\kbar$, and is not to be confused with $X_{\bar{\eta}}$.  This exact sequence can be used to give some criteria for the homomorphism $\H^2(k, \Pic \bar{B}) \to \H^2(k, \Pic \bar{X})$ to be injective. For example, the associated long exact sequence in Galois cohomology shows that this morphism is injective if $\H^1(k, \Pic \bar{X}_\eta) = 0.$
%Consider the associated long exact sequence
%\[
%\H^1(k, \Pic \bar{B}) \to \H^1(k, \Pic \bar{X}) \to \H^1(k, \Pic \bar{X}_\eta) \to \H^2(k, \Pic \bar{B}) \to \H^2(k, \Pic \bar{X})
%\]
%in Galois cohomology. If $\H^1(k, \Pic \bar{B})=0$ (e.g.~$B=\PP^m$), then we see that the injetivity of $\H^2(k, \Pic \bar{B}) \to \H^2(k, \Pic \bar{X})$ is equivalent to  for example, if $\Pic \bar{X}_\eta \cong \ZZ$. 
This injectivity also holds when the natural map $\Pic \bar{B} \to \Pic \bar{X}$ has a Galois-equivariant left inverse.  This is certainly true if $X \to B$ has a section, but can hold more generally, as we will see in the following proposition.

For the application to Theorem~\ref{thm:H1}, we take $B=\PP^m$.  Recall that, if $Y$ is a smooth projective variety, then a closed subvariety $X \subset Y$ of codimension $c$ is a \emph{complete intersection} if $X$ is the scheme-theoretic intersection of $c$ very ample divisors in $Y$.

\begin{proposition}\label{prop:ci}
Let $k$ be a number field.  Let $X \subset \PP^r \times \PP^m$ be a complete intersection of dimension $\ge 3$, and suppose that the projection $\pi \colon X \to  \PP^m$ satisfies Condition~\ref{cond}.   Then 
$\H^2(k, \Pic \bar{B}) \to \H^2(k, \Pic \bar{X})$ is injective and, furthermore, 
the residue map
\[
\H^1(K, \Pic X_{\bar{\eta}})[n] \xrightarrow{\oplus \rho_{d}} \bigoplus_{d \in B^{(1)}} \H^0(\ff{d}, \H^1(X_{\bar{d}}^\sm, \Zn)).
\]
of~\eqref{eq:H1picres} is also injective.
\end{proposition}
\begin{proof}
%We have $\Br \PP^m_\kbar=0$ and $\H^3(\PP^m_\kbar, \Gm)=0$.  
The Lefschetz hyperplane theorem for Picard groups gives $\Pic \bar{X} \cong \ZZ \times \ZZ$ with trivial Galois action, with $\pi^* \colon \Pic \bar{\PP}^m \to \Pic \bar{X}$ being the inclusion of one factor; it follows that $\H^2(k, \Pic \bar{\PP}^m) \to \H^2(k, \Pic \bar{X})$ is injective.  
 By \cite[Prop.~2.6]{damaris} and its proof it follows that 
 the natural map $\Br \PP^m \to \Br X$ is an isomorphism and 
$\Br \bar{X}=0$.
% Since $k$ is a number field, we have $\H^1(k, \Pic\bar{X}) \cong \Br X / \Br k = 0$.
%Also, the Lefschetz hyperplane theorem for \'etale cohomology shows that $\H^2(\PP^r_\kbar \times \PP^m_\kbar, \mmu_n) \to \H^2(\bar{X},\mmu_n)$ is an isomorphism for all $n$; it follows from the Kummer sequence that $\Br \bar{X}$ is trivial.  Combining this with the calculation $\H^1(k, \Pic \bar{X})=0$, the Hochschild--Serre spectral sequence then shows that $\Br k \to \Br X$ is an isomorphism.
%Thus the sufficient conditions of Corollary~\ref{cor:H1R1} are satisfied,
Thus Lemmas~\ref{lem:H1R1} and~\ref{lem:H3} show that $\H^1(B, \R^1 \pi_* \Gm)$ vanishes,
and Lemma~\ref{lem:kernel} gives the claimed result.
\end{proof}

\subsection{Application to elements of $\Br X_\eta$}\label{sec:Br}

Until now we have been concentrating on the case of Theorem~\ref{thm:H1} in which there is a non-trivial element of $\H^1(K, \Pic X_{\bar{\eta}})$.  The other case, in which there is a non-trivial element of $\Br X_\eta / \Br K$, is significantly easier, but fits into the same framework.  

To begin with, fix a positive integer $n$, and let $A$ be an $n$-torsion element of $\Br X_\eta$.  The Kummer sequence gives an exact sequence
\[
0 \to \Pic X_\eta / n \to \H^2(X_\eta, \mmu_n) \to \Br X_\eta[n] \to 0
\]
showing that $A$ may be lifted to $\H^2(X_\eta, \mmu_n)$.  Applying the natural map $\H^2(X_\eta, \mmu_n) \to \H^0(K, \H^2(X_{\bar{\eta}}, \mmu_n))$ gives a class $\alpha \in \H^0(K, \H^2(X_{\bar{\eta}}, \mmu_n))$.  
Let $d$ be any point of codimension $1$ in $B$.
It follows easily from the commutative diagram~\eqref{eq:comm1} in Proposition~\ref{prop:relpur}
that the relative residue $\rho_d(\alpha)$ coincides with the image of $A$ under the composition
\[
\Br X \xrightarrow{\partial} \H^1(\ff{X_d}, \QQ/\ZZ) \xrightarrow{\textrm{res}} \H^0(\ff{d},\H^1(\ff{X_{\bar{d}}}, \QQ/\ZZ))
\]
where $\partial$ is the usual residue map for the Brauer group associated to the codimension-1 point $X_d \in X^{(1)}$, as defined by Grothendieck, and $\textrm{res}$ is the restriction map in Galois cohomology.
(Here we consider $\H^1(X_{\bar{d}}^\ns,\Zn)$ as a subgroup of $\H^1(\ff{X_{\bar{d}}},\QQ/\ZZ)$ in the natural way.)  In particular, $\rho_d(\alpha)$ does not depend on how we lift $A$ to $\H^2(X,\mmu_n)$.

\begin{proposition}\label{prop:Br}
Suppose the conditions of Theorem~\ref{thm:H1} hold and that $\Pic X_{\bar{\eta}}$ is torsion-free.  Let $A$ lie in $\Br X_\eta[n]$ and   let $\alpha$ be constructed as above.  Then we have $\rho_d(\alpha)=0$ for all $d \in B^{(1)}$ if and only if $A$ lies in the image of $\Br K \to \Br X_\eta$.
\end{proposition}
\begin{proof}
Let us first prove that $\H^3(B,\mmu_n) \to \H^3(X,\mmu_n)$ is injective.  The Kummer sequence gives a commutative diagram as follows.
\[
\begin{CD}
0 @>>> \Br B/n @>>> \H^3(B,\mmu_n) @>>> \H^3(B,\Gm)[n] @>>> 0 \\
@. @VVV @VVV @VVV @. \\
0 @>>> \Br X/n @>>> \H^3(X,\mmu_n) @>>> \H^3(X,\Gm)[n] @>>> 0. \\
\end{CD}
\]
The assumptions $\Br \bar{X}=0$ and $\H^1(k,\Pic\bar{X})=0$ together imply
$\Br k = \Br B = \Br X$, so the left-hand vertical map is an isomorphism.
By Lemma~\ref{lem:H3}, the right-hand vertical map is injective, and therefore the middle one is also injective, as claimed

Now consider the commutative diagram
\[
\begin{CD}
@. \H^2(K,\mmu_n) \\
@. @VVV \\
\H^2(X,\mmu_n) @>>> \H^2(X_\eta, \mmu_n) \\
@V{f}VV @V{g}VV \\
\H^0(B,\R^2 \pi_* \mmu_n) @>>> \H^0(K, \H^2(X_{\bar{\eta}}, \mmu_n)) @>{\oplus \rho_d}>>
\bigoplus_{d \in B^{(1)}} \H^0(\ff{d}, \H^1(X_{\bar{d}}^\sm, \Zn))
\end{CD}
\]
in which the bottom row is the exact sequence of Proposition~\ref{prop:R2seq}, and the middle column is exact by the Leray spectral sequence and $\Pic X_{\bar{\eta}}[n]=0$.
Again using the Leray spectral sequence, the homomorphism $f$ fits into an exact sequence
\[
\H^2(X,\mmu_n) \xrightarrow{f} \H^0(B, \R^2 \pi_* \mmu_n) \to \H^3(B,\mmu_n) \to \H^3(X,\mmu_n)
\]
showing that $f$ is surjective.

Let $A$ be an $n$-torsion element in $\Br X_\eta$ and lift $A$ to a class $\beta \in \H^2(X_\eta,\mmu_n)$, so that $\alpha = g(\beta)$.  An easy diagram-chase now shows that $\rho_d(\alpha)=0$ holds for all $d$ if and only if $\beta$ lies in the subgroup of $\H^2(X_\eta,\mmu_n)$ generated by $\H^2(X,\mmu_n)$ and $\H^2(K,\mmu_n)$.  Using the Kummer sequence and $\Br X = \Br k$, we see that this holds if and only if $A$ lies in the image of $\Br K[n] \to \Br X_\eta[n]$, and in particular in the image of $\Br K$. 
\end{proof}

\subsection{Ramified elements usually obstruct weak approximation}\label{s:bull-rider}

The main aim of this section is to establish Proposition \ref{prop:H2}.
Throughout this section, we fix the following notation.  Let $k$ be a number field, and let $\pi \colon X \to B$ be a morphism of smooth varieties over $k$ satisfying Condition~\ref{cond}.  Let $U \subset B$ be a non-empty open subset over which $\pi$ is smooth and proper.  Fix an integer $n>1$.  Our aim is to show that a ramified element $\alpha$ of $\H^0(K, \H^2(X_{\bar{\eta}}, \mmu_n))$ obstructs weak approximation on ``most'' fibres $X_P$ for $P \in U(k)$.  Note that $\alpha$ could come from an $n$-torsion element in $\H^1(K, \Pic X_{\bar{\eta}})$, as described above, or equally well from an element of $\Br X_\eta[n]$.

We begin by showing how to specialise an element of $\H^0(K, \H^2(X_{\bar{\eta}}, \mmu_n))$ at a point $P \in U(k)$.

\begin{lemma}
The natural map $\H^0(U, \R^2 \pi_* \mmu_n) \to \H^0(K, \H^2(X_{\bar{\eta}}, \mmu_n))$ is an isomorphism.
\end{lemma}
\begin{proof}
Let $D$ be any prime divisor on $U$, with generic point $d$.  Then the fibre $X_d$ is smooth, and we have $\H^1(X_{\bar{d}},\mmu_n)=0$ by Lemma~\ref{lem:R1zero} and proper base change.  The sheaves $\Zn$ and $\mmu_n$ are isomorphic on $X_{\bar{d}}$, giving $\H^1(X_{\bar{d}},\Zn)=0$.  Thus Proposition~\ref{prop:R2seq} gives the claimed result.
\end{proof}

Thus any element $\alpha \in \H^0(K, \H^2(X_{\bar{\eta}}, \mmu_n))$ extends to $U$.  If $P \in U(k)$ is any rational point, then $\alpha$ can be specialised to an element of $\H^0(k, \H^2(\bar{X}_P, \mmu_n))$.  The following lemma shows that, under suitable hypotheses, we can then lift to $\H^2(X_P, \mmu_n)$.

\begin{lemma}\label{lem:XPsurj}
Let $P$ be a point of $U(k)$, and let $\pi_P \colon X_P \to P$ be the base change of $\pi$.  If $n$ is even, assume that the fibre $X_P$ has points in every real completion of $k$.  Let $\bar{X}_P$ denote the base change of $X_P$ to $\kbar$.  Then the homomorphism $\H^2(X_P, \mmu_n) \to \H^0(k, \H^2(\bar{X}_P,\mmu_n))$ is surjective, with kernel $\Br k[n]$.
\end{lemma}
\begin{proof}
By Lemma~\ref{lem:R1zero} and proper base change, we have $\H^1(\bar{X}_P, \mmu_n)=0$.  The Hochschild--Serre spectral sequence gives an exact sequence 
\[
0 \to \H^2(k, \mmu_n) \to \H^2(X_P, \mmu_n) \to \H^0(k, \H^2(\bar{X}_P, \mmu_n)) \to \H^3(k, \mmu_n) \to \H^3(X_P, \mmu_n).
\]
The Kummer sequence gives $\H^2(k,\mmu_n) \cong \Br k[n]$ (since $\H^1(k,\Gm)$ vanishes) and so the result will be established if we can show that $\H^3(k,\mmu_n) \to \H^3(X_P,\mmu_n)$ is injective.
But the natural map $\H^3(k,\mmu_n) \to \prod_{v\text{ real}} \H^3(k_v, \mmu_n)$ is an isomorphism (see, for example, \cite[Theorem~4.10(c)]{milneadt}).  So evaluating at a point in $X(k_v)$ for all real places $v$ gives a left inverse to $\H^3(k,\mmu_n) \to \H^3(X,\mmu_n)$, which is therefore injective.
\end{proof}

Under the conditions of Lemma~\ref{lem:XPsurj}, we can specialise a class $\alpha$ belonging to $\H^0(U, \R^2 \pi_* \mmu_n)$ to an element of $\Br X_P[n] $ using the sequence of homomorphisms
\[
\H^0(U, \R^2 \pi_* \mmu_n) \to \H^0(k, \H^2(\bar{X}_P, \mmu_n)) 
\leftarrow 
\H^2(X_P, \mmu_n) \to \Br X_P[n].
\]
The resulting class, which we will denote $\spec_P(\alpha)$, is determined only up to $\Br k[n]$.  (In fact the condition that $X_P$ be soluble at the real places is only for convenience:  in general, we can make the construction work but at the expense of possibly increasing $n$.  Since we are only interested in everywhere locally soluble fibres, we do not mind imposing this condition.)  It is straightforward to check that, if $\alpha$ comes from a class in $\H^1(K, \Pic X_{\bar{\eta}})$ as described in \S\ref{s:cow-herder}, then $\spec_P(\alpha)$ is the same as the class obtained by first specialising to $\H^1(k, \Pic \bar{X}_P)$ and then lifting to $\Br X_P / \Br k$.

The final ingredient needed in the proof of Proposition \ref{prop:H2} is a lemma concerning the existence of rational points on a family of torsors defined over $\fo$.

\begin{lemma}\label{lem:points}
Let $k$ be a number field, $\fo$ the ring of integers of $k$ and let $\pi \colon Z \to S$ be a dominant morphism of normal, integral, flat, separated $\fo$-schemes of finite type.  Denote by $\eta$ the generic point of $S$ and $\bar{\eta}$ a geometric point lying over $\eta$.  Suppose that $Z_{\bar{\eta}}$ is integral.  Let $\gamma \in \H^0(S, \R^1 \pi_* \Zn)$ be a class, and suppose that the restriction of $\gamma$ to $\H^1(Z_{\bar{\eta}}, \Zn)$ has order $n$.  For a geometric point $\bar{s} \in S$, let $Y(\bar{s})$ denote the torsor over $Z_{\bar{s}}$ defined by the restriction of $\gamma$ to $\H^1(Z_{\bar{s}}, \Zn)$.
Then there is are dense open subsets $S' \subset S$ and $U \subset \Spec \fo$ such that, 
for every $\fp \in U$ and any $\bar{s} \in S'(\overline{\FF}_\fp)$,
%for every closed point $s \in S'$ and geometric point $\bar{s}$ lying over $s$, 
any $\FF_\fp$-variety geometrically isomorphic to $Y(\bar{s})$ has a $\FF_\fp$-rational point.
%Then there are a dense open subset $S' \subset S$ and a bound $M>0$ such that, for any finite field $\FF$ of cardinality at least $M$ with algebraic closure $\bar{\FF}$, any geometric point $\bar{s} \colon \Spec \bar{\FF} \to S'$ and any variety $V$ over $\FF$ geometrically isomorphic to $Y(\bar{s})$, we have $V(\FF) \neq \emptyset$.
\end{lemma}
\begin{proof}
By the standard description of the stalks of $\R^1 \pi_* \Zn$, there is an \'etale morphism $T \to S$ such that the restriction of $\gamma$ to $\H^0(T, \R^1 \pi_* \Zn)$ lies in the image of $\H^1(Z_T, \Zn)$.  Replacing $T$ by a connected component, we may assume that $T$ is integral.  The image of $T \to S$ is a dense open subset, and any geometric point $\bar{s}$ lying in the image of $T \to S$ factors through $T$ (and this applies in particular to $\bar{\eta}$), so we may replace $S$ by $T$ and assume that $\gamma$ lifts to $\gamma' \in \H^1(Z, \Zn)$.

Let $f \colon Y \to Z$ be a torsor representing $\gamma'$.  For any geometric point $\bar{s} \in S$, the fibre $Y_{\bar{s}}$ is isomorphic to the torsor $Y(\bar{s})$ in the statement of the lemma.  We claim that the geometric generic fibre $Y_{\bar{\eta}}$ is integral.  The torsor $Y_{\bar{\eta}} \to Z_{\bar{\eta}}$ represents the image of $\gamma$ in $\H^1(Z_{\bar{\eta}}, \Zn)$, which by assumption has order $n$.
Because $Z$ is normal, it follows from Proposition~10.1 of~\cite[Expos\'e~1]{SGA1} that the natural map $\H^1(Z_{\bar{\eta}},\Zn) \to \H^1(\ff{Z_{\bar{\eta}}}, \Zn)$ is injective, and so the \'etale algebra $\ff{Y_{\bar{\eta}}}$ is a torsor over $\ff{Z_{\bar{\eta}}}$ represented by a cohomology class of order $n$, which by Kummer theory is a field.  Again by Proposition~10.1 of~\cite[Expos\'e~1]{SGA1} we see that $Y_{\bar{\eta}}$ is integral.  Now applying Lemma~\ref{lem:lw} to $Y \to S$ completes the proof.
\end{proof}

We are now in a position to prove the main proposition recorded in \S \ref{s:cow-herder}.

\begin{proof}[Proof of Proposition \ref{prop:H2}]
We continue to  denote the morphism $\sX \to \sB$ also by $\pi$.
To prove the proposition, we may replace $\sB$ with any open subset containing both $U$ and $d$.  In particular, we may (and do) assume that $n$ is invertible on $\sB$; that $\sB$ is flat over $\fo$ and regular; that $\sD$ is regular; and that $\pi$ satisfies Condition~\ref{cond}.  Applying Proposition~\ref{prop:R2seq}, we further shrink $\sB$ to ensure that $\alpha$ lies in $\H^0(\sB \setminus \sD, \R^2 \pi_* \mmu_n)$.  Finally, we replace $\sX$ by the open subscheme on which $\pi$ is smooth, so that (by diagram~\eqref{eq:dD}) the relative residue $\rho_d(\alpha)$ extends to $\rho_\sD(\alpha) \in \H^0(\sD, \R^1 (\pi_\sD)_* \Zn)$.  

Since $X_d$ is geometrically integral, there is a dense open subset $\sV \subset \sD$ such that the geometric fibres of $\pi$ over $\sV$ are integral.
If $s \in \sV$ is a closed point, and $\bar{s}$ a geometric point lying over $s$, then $\rho_\sD(\alpha)$ restricts to an element of $\H^0(s, \H^1(\sX_{\bar{s}},\Zn))$.  Because $\sX_{\bar{s}}$ is geometrically integral, we have $\H^0(\sX_{\bar{s}},\Zn) = \Zn$.  The Hochschild--Serre spectral sequence gives an exact sequence
\[
0 \to \H^1(s, \Zn) \to \H^1(\sX_s, \Zn) \to \H^0(s, \H^1(\sX_{\bar{s}}, \Zn)) \to \H^2(s, \Zn).
\]
The residue field $\ff{s}$ is finite, so we have $\H^2(s, \Zn)=0$.  Hence the restriction of $\rho_\sD(\alpha)$ is represented by a torsor $\sY(s) \to \sX_s$, where the choice of $\sY(s)$ is unique only up to twisting by an element of $\H^1(s,\Zn)$.  
It follows from Lemma~\ref{lem:points} that we can shrink $\sV$ to ensure that, for every closed point $s$ of $\sV$, every $\ff{s}$-twist of the variety $\sY(s)$ has a $\ff{s}$-rational point.

Now let us show that $\sV$ has the property claimed in the proposition.  Let $P$ be a point of $U(k)$, and suppose that its Zariski closure $\sP$ meets $\sV$ transversely at a closed point $s$.  Let $\p$ be the prime ideal of $\fo$ over which $s$ lies.  The argument that follows is local at $s$, so we replace $\sB$ by the local scheme $\Spec \OO_{\sB,s}$, on which $\sV=\sD$ is a regular divisor.
Let $k_\p$ be the completion of $k$ at $\p$ and $\fo_\p$ the completion of $\fo$.  Let $\hat{X}_P$ denote the base change of $X_P$ to $k_\p$ and let $\hat{\sX}_\sP$ denote the base change of $\sX_\sP$ to $\fo_\p$.
We have a commutative diagram with exact rows as follows.
\[
\begin{CD}
\H^0(\sB, \R^2 \pi_* \mmu_n) @>>> \H^0(\sB \setminus \sD, \R^2 \pi_* \mmu_n) @>{\rho_\sD}>> \H^0(\sD, \R^1 (\pi_\sD)_* \Zn) \\
@VVV @VVV @VVV \\
\H^0(\sP, \R^2 (\pi_\sP)_* \mmu_n) @>>> \H^0(k, \H^2(\bar{X}_P, \mmu_n)) @>{\rho_s}>> \H^0(\ff{s}, \H^1(\sX_{\bar{s}}, \Zn)) \\
@AAA @AAA @AAA \\
\H^2(\sX_\sP, \mmu_n) @>>> \H^2(X_P, \mmu_n) @>>> \H^1(\sX_s, \Zn) \\
@VVV @VVV @| \\
\Br(\sX_\sP)[n] @>>> \Br (X_P)[n] @>>> \H^1(\sX_s, \Zn) \\
@VVV @VVV @| \\
\Br(\hat{\sX}_\sP)[n] @>>> \Br (\hat{X}_P)[n] @>{\partial_{\sX_s}}>> \H^1(\sX_s, \Zn).
\end{CD}
\]
In this diagram, the top row comes from Proposition~\ref{prop:relpur}(\ref{relpur1}) applied to $\sD \subset \sB$; the next two rows come from Proposition~\ref{prop:relpur}(\ref{relpur1}) applied to $s \subset \sP$; and the bottom two rows come from the purity theorem from the Brauer group on $\sX_\sP$ and $\hat{\sX}_\sP$ respectively (see, for example, \cite[Corollary~2.5]{bright'}).
The class $\alpha$ lies in the middle top group, and the class $\spec_P(\alpha)$ (defined modulo $\Br k[n]$) lies in $\Br(X_P)[n]$.  Denote by $\hat{\alpha}$ the image of $\spec_P(\alpha)$ in $\Br(\hat{X}_P)[n]$.  Because the diagram commutes, the class $\partial(\hat{\alpha})$ is the class of our torsor $\sY(s)$, again only defined modulo twists.

Lemma~5.12 of~\cite{bright'}, in our notation, states the following.  To every class $r \in \Br k_\p[n]$ we may associate a certain twist of $\sY(s)$, which we will denote $\sY(s)^{[r]}$.  Then $r$ lies in the image of the evaluation map $X_P(k_\p) \to \Br k_\p[n]$ coming from $\hat{\alpha}$ if and only if $\sY(s)^{[r]}$ has a $\ff{s}$-rational point.  By construction of $\sV$, this is true for all $r \in \Br k_\p[n]$, so the evaluation map is surjective.
\end{proof}

\section{An application of the large sieve}\label{s:sieve}

This section contains the analytic part of the proof of Theorem \ref{thm:H1}, an overview of which is given in \S \ref{s:cow-herder}. Our main goal is to establish Proposition 
\ref{cor:sieve}.

Let $k$ be a number field of degree $d$ over $\QQ$, with associated 
ring of integers $\fo$, and   let $H: \PP^n(k)\to \RR_{\geq 1}$ be the  height function that was constructed in \S \ref{s:goat-herder}. For given $M\in \NN$ and non-constant homogeneous 
polynomials  $f,g\in \fo[X_0,\dots,X_n]$, we shall need to  study the counting function
\begin{equation}\label{eq:sheep-herder}
N(T;B,M)=\#\left\{x\in T : 
\begin{array}{l}
H(x)\leq B\\
\text{$\exists$  $\Norm\fp >M$ s.t. $\fp\|f (\x)$ and $\fp\nmid g(\x)$}
\end{array}{}
\right\},
\end{equation}
for any non-empty subset $T\subset \PP^n(k).$
In the definition 
of $N(T;B,M)$ the main constraint is to be understood as there exists a prime 
ideal $\fp$ such that $\n\fp>M$ and a primitive representative 
$\x=(x_0,\dots,x_n)\in (\fo/\fp^2)^{n+1}$ of $x \bmod \fp^2$ such that 
$\fp\| f(\x)$  and $\fp\nmid g(\x)$. Here we write  $\fp\| f(\x)$ to mean that $\fp\mid f(\x)$ but $\fp^2\nmid f(\x)$.

The following is the main result of this section.

\begin{proposition}\label{prop:fly-catcher}
Assume that $f \nmid ag$ for all $a \in \fo$ and that $f\neq f_1f_2^2$ 
for polynomials $f_1,f_2\in \fo[X_0,\dots, X_n].$
%is not the multiple of a  square of a polynomial defined over $\fo$.
Then for fixed $M\in \NN$ and any non-empty  subset $T\subset \PP^n(k)$ we have 
$$
\#\{x\in T: H(x)\leq B\}-N(T;B,M)\ll_{M,f,g}   \frac{B^{n+1}}{\log B}.
$$
The implied constant is allowed to depend on $\ve,M,f$ and $g$.
\end{proposition}

If $f\mid ag$ for some $a \in \fo$, or if $f=f_1f_2^2$ for polynomials
$f_1,f_2\in \fo[X_0,\dots,X_n]$, then it is clear that $N(T;B,M)=0$. Thus the hypotheses of the proposition are both necessary and sufficient to draw the conclusion.
For comparison we note that when $T=U(k)$ for some non-empty open subset $U\subset \PP^n$, the counting function
 $\#\{x\in T: H(x)\leq B\}$ has exact order $B^{n+1}$, 
 which exceeds the upper bound recorded in Proposition \ref{prop:fly-catcher}.
 Before proving this result, let us see how it can be used to 
establish Proposition \ref{cor:sieve}.

\begin{proof}[Proof of Proposition \ref{cor:sieve}]
The generic fibre $\sD \otimes_\fo k$ is a reduced hypersurface in $\PP^n$, so is defined by a single square-free homogeneous polynomial $f \in k[X_0, \dotsc, X_n]$.  Clearing denominators, we may ensure that $f$ lies in $\fo[X_0, \dotsc, X_n]$.  Then the subscheme of $\PP^n_\fo$ defined by $f$ consists of $\sD$ together with the whole of $\PP^n_{\FF_\fp}$ for each prime $\p$ dividing $f$.  (If $\fo$ is a PID, then we can take $f$ to be primitive, so that $\sD$ is defined by the single polynomial $f$, but in general this may not be possible.)  Set $M = \max_{\p \mid f} \Norm \p$.
Let $S$ be the complement of $\sV$ in $\sD$, and let $g \in \fo[X_0, \dotsc, X_n]$ be any homogeneous polynomial vanishing on $S$ but not identically zero on $\sD$.  (In the case $\sV=\sD$, take $g$ to be any non-constant homogeneous polynomial not zero on $\sD$.)

Let  $P \in U(k)$. 
Suppose that there exists a prime  $\p$ with $\Norm \p > M$ 
and a  representative $\x = (x_0, \dotsc, x_n) \in (\fo/\fp^2)^{n+1}$ of $P$ satisfying $\p \| f(\x)$ and $\p \nmid g(\x)$.  Note that the latter condition implies that the $x_i$ are not all divisible by $\p$.  We claim that $\bar{P}$ meets $\sV$ transversely in the fibre above $\p$.  Indeed, let $s$ be the closed point where $\bar{P}$ meets the fibre above $\p$; explicitly, $s$ is defined by the ideal $I = (\p,X_0-x_0, \dotsc, X_n-x_n)$.  The condition $\p \mid f(\x)$ shows that $f$ lies in $I$; together with the condition $\Norm\p>M$, this shows that $s$ lies in $\sD$.  The condition $\p \nmid g(\x)$ shows that $s$ lies in $\sV$.  The ideal generated by $f$ and the defining equations for $P$ is $J=(f(\x), X_0-x_0, \dotsc, X_n-x_n)$, which is contained in $I$.  The condition $\p \| f(\x)$ shows that, on the open neighbourhood of $s$ obtained by inverting all other primes dividing $f(\x)$, $I$ coincides with the ideal defining $s$. That is, $\bar{P}$ and $f$ meet transversely at $s$. This establishes the claim.

It therefore  follows that 
\begin{align*}
0&\leq \#\{P\in T:H(P)\leq B\}-
\#\{P\in T_{\text{trans}}:H(P)\leq B)\}\\
&\leq \#\{P\in T:H(P)\leq B\}-
N(T;B,M),
\end{align*}
in the notation of 
\eqref{eq:sheep-herder}. 
An application of Proposition 
\ref{prop:fly-catcher} therefore completes the proof.
\end{proof}

In order to establish Proposition  \ref{prop:fly-catcher} we will produce an upper bound for the number of $x\in T$ with $H(x)\leq B$ for which there is no prime ideal $\fp$ of norm $\n\fp>M$ such that 
$\fp\| f(\x)$ and $\fp\nmid g(\x)$.
Let us write $N^\circ(T;B,M)$ for this  complementary cardinality.
Since we are only interested in an upper bound for $N^\circ(T;B,M)$ we may drop any of the defining conditions that we care to. In particular we henceforth ignore the constraint that the points counted  by 
$N^\circ(T;B,M)$ must lie in $T$.
%As in the proof of Proposition \ref{prop:sieve_of_Ekedahl_proj}, any $x\in \PP^n(k)$ has 
%representative  coordinates $\x\in \mathfrak{o}^{n+1}$ such that the $\mathfrak{o}$-span 
% $\langle x_0 , \dots , x_n \rangle$ is one of the representative ideals $\mathfrak{c}_i$ for the class group of $k$.   
Define the  distance function   $\|\x\|_{\star}:k^{n+1}\to \RR_{\geq 0}$ via
$$
\|\textbf{x} \|_{\star} = 
 \sup_{\substack{
 0\leq i\leq n\\
 \nu\mid\infty}} \| x_i \|_\nu,
$$
where $\|\cdot\|_\nu$ is the normalised absolute value associated to the place $\nu$. For any $H\geq 1$ we introduce the set
$$
Z_{n+1}(H)=\{\x\in \fo^{n+1}: \|\x\|_\star\leq H\}.
$$
%$$
%Z_{n+1}(H)=\{\x\in \fo^{n+1}: \text{$\exists~i\in \{1,\dots,h\}$ s.t. $\langle x_0,\dots,x_n\rangle=\mathfrak{c}_i$}, ~
%\|\x\|_\star\leq H\}.
%$$
As explained in  \cite[Prop.~3]{Broberg}, a  consequence of Dirichlet's unit theorem 
is the existence of a constant $c_1>0$, depending only on $n$, such that 
any  $x\in \mathbb{P}^n(k)$ with $H(x)\leq B$ has a representative 
$\x\in Z_{n+1}(c_1B)$. (Note, however, that elements of the latter set do not uniquely determine elements of the former.)

Our work so far shows that 
$
N^\circ(T;B,M)\leq \sum_{i=1,2}N_i(B,M),
$
where 
$$
N_1(B,M)=
\#\{\x\in Z_{n+1}(c_1B): \text{$\fp^2\mid f(\x)$ for all $\Norm\fp >M$ s.t. $\fp\mid f (\x)$}\}
$$
and 
$$
N_2(B,M)=
\#\{\x\in Z_{n+1}(c_1B): \text{$\fp\mid g(\x)$ for all $\Norm\fp >M$ s.t. $\fp\mid f (\x)$}\},
$$
with $c_1$ as above. For fixed $M$, we need  upper bounds for 
$N_1(B,M)$ and $N_2(B,M)$ which agree with the bound recorded in Proposition \ref{prop:fly-catcher}.
Our primary tool in this endeavour 
is the following multi-dimensional version of the arithmetic large sieve inequality.

\begin{lemma}\label{lem:large_sieve}
Let $m,n\in \NN$,  let $B\geq 1$ and 
let $\Sigma\subset \fo^{n+1}$ be any subset.
For each prime ideal $\fp$ suppose there exists  $\omega(\fp)\in [0,1]$ such that    
the image of $\Sigma$ in $(\fo/\fp^m)^{n+1}$ has at most $(\Norm\fp)^{m(n+1)}(1-\omega(\fp))$ elements. Then 
there is a constant $c_k>0$ such that 
$$
\#\{\x\in \Sigma: \|\x\|_\star\leq B\} \leq c_k\frac{B^{n+1}}{L(B^{1/(2m)})},
$$
with
$$
L(z)=\sum_{\fa }\prod_{\fp\mid \fa} \frac{\omega(\fp)}{1-\omega(\fp)},
$$
where the sum is over square-free integral ideals $\fa\subset \fo$ such that $\Norm\fa\leq z$.
\end{lemma}

\begin{proof}
When $k=\QQ$ this is recorded in 
\cite[\S 6]{serre}. The extension to general number fields is  standard,
and follows the method given in \cite[Chap.~12]{serre-book}.
\end{proof}

Before beginning our estimation of 
$N_1(B,M)$ and $N_2(B,M)$, we may clearly assume without loss of generality that $f$ is square-free and that $f$ and $g$ share no common factors.

%T [lots of edits from here to end..

\subsection{Estimating $N_1(B,M)$} 
Let $\Sigma$ denote the set 
of elements $\x\in Z_{n+1}(c_1B)$ such that $\fp^2\mid f(\x)$ for all $\Norm\fp >M$ with $\fp\mid f (\x)$.
For any prime ideal $\fp$, let $\Sigma_\fp$ denote the image of $\Sigma$ in $R_\fp=(\fo/\fp^2)^{n+1}$.
Put $\Sigma_\fp^\circ=R_\fp\setminus \Sigma_\fp$. 
If $\Norm \fp\leq M$ then $\Sigma_\fp=R_\fp$. If $\Norm \fp>M$ then $\Sigma_\fp^\circ$ is the set of $\y\in R_\fp$ for which $\fp\| f(\y)$.
Thus we have 
$$
\#\Sigma_\fp^\circ=\#\{\y\in R_\fp: \fp\mid f(\y)\} - \#\{\y\in R_\fp: \fp^2\mid f(\y)\},
$$
for $\Norm \fp> M$.
Possibly after
enlarging $M$,
for any prime ideal $\fp$ with  
$\Norm\fp>M$  
there exists a factorisation   $f=f_1\dots f_r$  over $\fo/\fp^2$ such that the following holds:
\begin{itemize}
\item
for  each $i\neq j$   the variety $f_i=f_j=0$ in $\PP_{\FF_\fp}^n$ has codimension 2 when viewed over $\FF_\fp=\fo/\fp$; and
\item
for each $i$ 
the reduction modulo $\fp$ of the variety 
cut out by the equations
$f_i=0$ and $\nabla f_i=\mathbf{0}$  has codimension at least $2$
in  $\PP_{\FF_\fp}^n$.
\end{itemize}
%Note that  $\nabla f_i$  not identically zero for each index $i$.
We proceed by establishing the following result.

\begin{lemma}\label{lem:cave}
We have $\#\{\y\in R_\fp: \fp^2\mid f(\y)\}\ll_{f} (\Norm \fp)^{2n}$
when $\n\fp>M$.
\end{lemma}

\begin{proof}
Assume that $\n\fp>M$.
Now $\fp^2\mid f(\y)$ if and only if either 
$\fp^2\mid f_i(\y)$ for some index $i$, or else 
$\fp\mid f_i(\y)$ and $\fp\mid f_j(\y)$ for a pair of indices $i\neq j$.
Choose an element $\alpha\in \fo$ such that $\mathrm{ord}_\fp(\alpha)=1$. 
Then elements of $R_\fp$ are in bijection with elements
$\u+\alpha\v$ for $\u,\v\in (\fo/\fp)^{n+1}=\FF_\fp^{n+1}$. 
Since $f_i=f_j=0$ 
cuts out a codimension 2 variety in $\PP_{\FF_\fp}^{n}$, the Lang--Weil \cite{LW54} estimate shows that 
$$
\#\{\y\in R_\fp: 
\text{$\fp\mid f_i(\y)$ and $\fp\mid f_j(\y)$}\}=
\quad
\sum_{\mathclap{\substack{\u\in \FF_\fp^{n+1}\\
\text{$\fp\mid f_i(\u)$ and $\fp\mid f_j(\u)$}}}}
\# \FF_\fp^{n+1} \,\, \ll_f (\n\fp)^{2n},
$$
since 
\begin{equation}\label{eq:pigeon-fancier}
\#\{\u\in \FF_\fp^{n+1}:
f_i(\u)\equiv f_j(\u)\equiv 0\bmod{\fp}\}\ll_f (\n\fp)^{n-1}
\end{equation}
if $i\neq j$. This is satisfactory for the lemma.
On the other hand, we use the 
  Taylor expansion to  deduce that 
\begin{align*}
\#\{\y\in R_\fp: \fp^2\mid f_i(\y)\}= \,
\sum_{\mathclap{\substack{\u\in \FF_\fp^{n+1}\\
f_i(\u)\equiv 0\bmod{\fp}}}
} D(\u),
\end{align*}
where
$D(\u)$ is the number of 
$\v\in \FF_\fp^{n+1}$ such that 
$$f_i(\u)+\alpha \v.\nabla f_i(\u) \equiv 0\bmod{\fp^2}.
$$
If $\fp\nmid \nabla f_i(\u)$ then there are $(\n\fp)^{n}$ choices for $\v$, and the 
Lang--Weil estimate yields the required bound. If $\fp\mid \nabla f_i(\u)$ however,
then the $\u$ are constrained to lie on a
variety of codimension at least $2$ in $\PP_{\FF_\fp}^n$. Hence a further application 
of the Lang--Weil estimate combined with the trivial bound $D(\u)\leq (\n\fp)^{n+1}$
concludes the proof.
\end{proof}

Now it follows from the Lang--Weil estimate that 
\begin{equation}\label{eq:swallet}
\#\{\y\in R_\fp: \fp\mid f(\y)\}\geq (\n\fp)^{n+1}\left((\n\fp)^n+O_f((\n\fp)^{n-\frac{1}{2}})\right).
\end{equation}
Hence it follows from 
this and Lemma \ref{lem:cave} that 
$\#\Sigma_\fp^\circ\geq (\n\fp)^{2n+1}+O_f((\n\fp)^{2n+\frac{1}{2}})$.
But this implies that  $\#\Sigma_\fp\leq (\n\fp)^{2(n+1)}(1-\omega(\fp))$, with 
\begin{equation}\label{eq:omega}
\omega(\fp)=
\begin{cases}
0, &\mbox{$\Norm\fp\leq M$,}\\
(\n\fp)^{-1}+O_f((\n\fp)^{-\frac{3}{2}}),
&\mbox{$\Norm\fp> M$.}
\end{cases}
\end{equation}
With this choice we clearly have 
\begin{equation}\label{eq:L}
\begin{split}
L(z)
&\geq 
\sum_{\substack{\n\fa\leq z\\ \text{$\fa$ square-free}\\
\fp\mid \fa \Rightarrow \n\fp>M
}}
\frac{1}{\n\fa}
\prod_{\substack{\fp\mid \fa}} \left(1
+O_f\left(\frac{1}{(\n\fp)^{\frac{1}{2}}}\right)\right)
\gg_{M,f} \log z ,
\end{split}
\end{equation}
where we emphasise that the implied constant is allowed to depend on $M$.
Invoking Lemma \ref{lem:large_sieve}, we have therefore established  the satisfactory bound
$$
N_1(B,M)\ll \frac{B^{n+1}}{L(B^{1/4})}\ll_{M,f} \frac{B^{n+1}}{\log B}.
$$

\subsection{Estimating $N_2(B,M)$}

This follows a similar pattern to before, but is simpler.  
By enlarging $M$, if necessary, we may assume that the variety $f=g=0$ 
in $\PP_{\FF_\fp}^n$ has codimension  $2$
for any prime ideal  $\fp$  with $\n\fp>M$.
In the present situation we only need to work with $m=1$ in the large sieve.
Let  $\Sigma$ be the set 
of $\x\in Z_{n+1}(c_1B) $ such that $\fp\mid g(\x)$ for all $\Norm\fp >M$ with $\fp\mid f (\x)$, and let $\Sigma_\fp$ be the image in $(\fo/\fp)^{n+1}$ for any prime ideal $\fp$.  Then  
$\Sigma_\fp^\circ=(\fo/\fp)^{n+1}\setminus \Sigma_\fp$ is the set of 
$\y\in (\fo/\fp)^{n+1}$ for which $\fp\mid f(\y)$ but $\fp\nmid g(\y)$
It follows from  \eqref{eq:pigeon-fancier} and \eqref{eq:swallet} 
that 
$\#\Sigma_\fp^\circ \geq 
 (\n\fp)^{n}+O_f((\n\fp)^{n-\frac{1}{2}})$.
Hence $\#\Sigma_\fp\leq (\n\fp)^{{n+1}}(1-\omega(\fp))$, with $\omega(\fp)$ as in 
\eqref{eq:omega}. In particular we still have the lower bound \eqref{eq:L}, whence  the large sieve 
yields
$
N_2(B,M) \ll_{M_f,g} B^{n+1}/(\log B).
$
This too is satisfactory and so completes the proof of Proposition \ref{prop:fly-catcher}. \qed

\end{document}